\documentclass[11pt]{article}
\usepackage{color}
\usepackage{amsmath}
\usepackage{amssymb}
\usepackage{amsthm}
\usepackage{epstopdf}
\usepackage{graphicx}
\usepackage{mathtools}
\usepackage[usenames,dvipsnames, table]{xcolor}
\usepackage[hang]{caption}
\usepackage{hyperref}
\usepackage{oldgerm}  
\usepackage[yyyymmdd,hhmmss]{datetime}

\usepackage{mathabx} 
\def\bee{\begin{enumerate}}\def\eee{\end{enumerate}}
\def\bei{\begin{itemize}}\def\eei{\end{itemize}}
\newcommand{\nco}{\newcommand}
\def\R{\mathbb{R}}\def\C{\mathbb{C}}
\nco{\red}{\color{red}}
\nco{\blue}{\color{blue}}
\nco{\cyan}{\color{cyan}}
\nco{\brown}{\color{Magenta}}

\nco{\magenta}{\color{magenta}}

\nco{\violet}{\color{violet}}
\nco{\orange}{\color{orange}}
\nco{\redend}{\normalcolor}
\nco{\blueend}{\normalcolor}
\def\inv#1{\frac{1}{#1}}
\def\tr{{\rm tr}\,}
\def\ommit#1{{}}
\def\({\left(}
\def\){\right)}
\def\ie{{\it i.e.,\/}\ }
\def\ie{{\rm i.e.,\/}\ }

\nco{\rnc}{\renewcommand}
\rnc{\title}[1]{{\Large\bf\mbox{}\\\medskip#1\bigskip\medskip\\}}
\rnc{\author}[1]{{\large #1\smallskip\\}}
\nco{\address}[1]{{\em #1\medskip\\}}

\newcommand\Sf{s\!f}
\def\diag{{\rm diag \,}}
\def\ii{\mathrm{i\,}}
\nco{\bun}{{\bf 1}}
\def\be{\begin{equation}}\def\ee{\end{equation}}
\def\bea{\begin{eqnarray}}\def\eea{\end{eqnarray}}
\def\bee{\begin{enumerate}}\def\eee{\end{enumerate}}
\def\bei{\begin{itemize}}\def\eei{\end{itemize}}
\def\oh{\frac{1}{2}}
\def\ommit#1{{}}

\def\SU{{\rm SU}}\def\U{{\rm U}}
\def\inv#1{\frac{1}{#1}}

\def\tr{{\rm tr\, }}

\def\eq=#1{\buildrel #1 \over{=}}
\def\CA{{\mathcal A}}
\def\CH{{\mathcal H}}     \def\CJ{{\mathcal J}} \def\CK{{\mathcal K}}  \def\CO{{\mathcal O}}
\def\CP{{\mathcal P}}\def\CV{{\mathcal V}}
\def\bH{\mathbf{H}}
\def\bH{\widetilde{\mathbf{H}}}
\def\diag{{\rm diag \,}}
\def\ii{\mathrm{i\,}}

\def\Z{\mathbb{Z}}

\def\R{\mathbb{R}}
\def\C{\mathbb{C}}
\def\T{\mathbb{T}}
\newtheorem{theorem}{Theorem}
\newtheorem{lemma}{Lemma}
\newtheorem{conjecture}{Conjecture}
\newtheorem{corollary}{Corollary}
\newtheorem{definition}{Definition}
\newtheorem{proposition}{Proposition}

\begin{document}
\begin{titlepage}
\begin{center}
\title{From orbital measures  \\
to Littlewood-Richardson coefficients\\ and hive polytopes}
\medskip
\author{Robert Coquereaux} 
\address{Aix Marseille Univ, Universit\'e de Toulon, CNRS, CPT, Marseille, France\\
robert.coquereaux@gmail.com}
\bigskip\medskip
\centerline{and}
\medskip
\author{Jean-Bernard Zuber}
\address{
 Sorbonne Universit\'e, UPMC Univ Paris 06, UMR 7589, LPTHE, F-75005, 
Paris, France\\
\& CNRS, UMR 7589, LPTHE, F-75005, Paris, France\\jean-bernard.zuber@upmc.fr
 }

\begin{abstract}
{The volume of the hive polytope  (or polytope of honeycombs) associated with a Littlewood-Richardson coefficient of SU($n$), or with a given admissible triple of highest weights, is expressed, in the generic case, in terms  of the Fourier transform of a convolution product of orbital measures.
Several properties of this function ---a function of three non-necessarily integral weights or of three multiplets of real eigenvalues for the associated Horn problem--- are already known.
In the integral case it can be thought of as a semi-classical approximation of Littlewood-Richardson coefficients.
We prove that it may be expressed as a local average of a finite number of such coefficients.
We also relate this function to the Littlewood-Richardson polynomials (stretching polynomials)  \ie  to the Ehrhart polynomials of the relevant hive polytopes. 
Several SU($n$) examples,  for $n=2,3,\ldots,6$, are explicitly worked out.}
\end{abstract}
\end{center}

\\


\end{titlepage}

\section*{Introduction} 
 In a previous paper \cite{JB-I}, 	 the following classical Horn's problem was addressed.
For two $n$ by $n$ Hermitian matrices $A$ and $B$ independently and uniformly 
distributed on their respective  unitary coadjoint orbits ${\mathcal O}_\alpha$ and   ${\mathcal O}_\beta$, labelled by their eigenvalues $\alpha$ and $\beta$, call $p(\gamma \vert \alpha, \beta)$
the  probability distribution function (PDF) of the eigenvalues $\gamma$ of their sum $C =A+B$. 
With no loss of generality, we assume throughout this paper that these eigenvalues are ordered, 
\be\label{orderalpha} \alpha_1\ge \alpha_2\ge \cdots\ge \alpha_n\ee 
and likewise for $\beta$ and $\gamma$.
In plain (probabilistic) terms, $p$ describes the conditional probability of  $\gamma$, given $\alpha$ and $\beta$. 
The general expression of $p$ was given in \cite{JB-I} in terms of orbital integrals and computed explicitly for low values of $n$.

 The aim of the present paper  is to study  the relations between  this function  $p$,
	and the tensor product multiplicities for irreducible representations (irreps) of the Lie groups $\U(n)$ or $\SU(n)$,  encoded by the Littlewood-Richardson (LR) coefficients. 
	
	Our main results are the following. 
	A central role is played by a function $\CJ_n(\alpha,\beta;\gamma)$ proportional to $p$,
		times a ratio of Vandermonde determinants, see (\ref{pSUn}).
This $\CJ_n$ is identified with the volume of the hive polytope
 (also called polytope  of honeycombs) associated with the triple 
$(\alpha,\beta;\gamma)$, see Proposition \ref{Inasavol}. It is thus known \cite{Heck82}  to provide the asymptotic behavior
of LR coefficients, for large weights. We find a relation between $\CJ_n$ and a sum of
LR coefficients over a {\it local}, finite, $n$-dependent, set of weights, which holds true irrespective of the 
asymptotic limit, see Theorem \ref{CI-LR}. In particular for SU(3), the sum is trivial and enables one
to express the LR coefficient as a piecewise linear function of the weights, see Proposition \ref{CI3}
and Corollary \ref{LRSU3}. 
 Implications on the stretching polynomial (sometimes called Littlewood-Richardson polynomial) and its coefficients are then investigated. 
 \\[5pt]
The content of this paper is as follows.
In sec.\,\ref{Conv}, we recall some basic facts on the geometric setting and on  tensor {\sl and} hive polytopes. We also
 collect formulae and results obtained in \cite{JB-I} on the function $\CJ_n$. Section 2 is
devoted to the connection between Harish-Chandra's orbital integrals and $\SU(n)$ character formulae,
to its implication on the relation between $\CJ_n$ and LR coefficients (Theorem \ref{CI-LR}), and
to consequences of the latter. In sec. 3, we reexamine the interpretation of $\CJ_n$ as the volume
of the hive polytope  in the generic case (Proposition 4),  through the analysis of the asymptotic regime.
 In the last section (examples), we take $n=2,3\ldots,6$, consider for each case the expression obtained for $\CJ_n$, give the local relation existing between the latter and LR coefficients (this involves two polynomials, 
that we call $R_n$ and $\widehat R_n$,  expressed as
characters of $\SU(n)$), and study the corresponding stretching polynomials. 
Some of the features studied in the main body of this article are finally illustrated in the last subsection where we consider a few specific hive polytopes.

\section{Convolution of orbital measures, density function and polytopes}
\label{Conv}

\subsection{Underlying geometrical picture}

	We consider a particular Gelfand pair $(\U(n) \ltimes H_n, \U(n))$ associated with the group action of  the Lie group $\U(n)$ on the vector space of $n$ by $n$ Hermitian matrices. 
	This geometrical setup allows one to develop a kind of harmonic analysis where ``points'' are replaced by coadjoint orbits of $\U(n)$ : the Dirac measure (delta function at the point $a$)
	is replaced by an orbital measure whose definition will be recalled below, and its Fourier transform, here an orbital transform, is given by the so-called  Harish-Chandra orbital function.
	This theory of  integral transforms can also be considered as a generalization of the usual Radon spherical transform (also called Funk transform).
	Contrarily to Dirac measures, orbital measures are not discrete, since their supports are orbits of the chosen Lie group. 
	Such a measure is described by a probability density function (PDF),  which is its Radon-Nikodym derivative with respect to the Lebesgue measure.
	\\
	In Fourier theory one may consider the measure formally defined as a convolution product of Dirac masses:
	 $<\delta_a \star \delta_b , f>  = \int \delta_{a+b}(x) f(x) dx$.
	Here we shall consider, instead, the convolution product of two orbital measures described by the orbital analog of $\delta_{a+b}(c)$, a probability density function 
	labelled by three $\U(n)$ orbits of $H_n$. 
	 These orbits and that function $p$ may be considered as  functions of three Hermitian matrices (we shall write it $p(C \vert A,B)$), and this 
	 answers a natural question in the context of the classical Horn problem,
	as mentioned above in the Introduction, see also sec.\,\ref{convorb} below.
	  This was spelled out in paper \cite{JB-I}.
	Our main concern, here, is the study of  the relations that exist between  this function  $p$,
	and the tensor product multiplicities for irreducible representations (irreps) of the Lie groups $\U(n)$ or $\SU(n)$,  encoded by the Littlewood-Richardson (LR) coefficients $N_{\lambda \mu}^{\nu}$.
	For small values of $n$ the function $p$ can be  explicitly calculated; for integral 	values of its arguments, the related function $\CJ_n$ can be considered as a semi-classical approximation of the LR coefficients.

\subsubsection{Orbital measures} 

For $F$, a function on the space of orbits, and {${\mathcal O}_A$}, the orbit going through  $A \in H_n$,  
one could formally consider the ``delta function''  $
<\delta_{{\mathcal O}_A} , F> = F({\mathcal O}_A)$,
 but we shall use test functions defined on $H_n$ instead. 

The orbital measure $m_A$, that plays the role of $\delta_{{\mathcal O}_A}$,  is therefore defined, for any continuous function $f$ on $H_n$, by 
 \[<m_A, f> = \int_{\U(n)} f(u^\star A u) du \] 
 where the integral is taken with respect to the Haar mesure\footnote{{In practice we use the 
normalized Haar measure that makes the volume of $\U(n)$ equal to $1$.}} on $\U(n)$, \ie by averaging the function $f$ on a $\U(n)$ coadjoint orbit.
	
\subsubsection{Fourier transform of orbital measures}
		Despite the appearance of the Haar measure on the group $\U(n)$ entering the definition of 
		$m_A$,  one should notice that  this is a measure on the vector space $H_n$, an abelian group.
		Being an analog of the Dirac measure, its orbital transform\footnote{The context being specified, people often simply write ``Fourier transform'' or ``Fourier orbital transform'' rather than ``spherical transform'' or ``orbital transform''.} is  a complex-valued function  $\widehat{m_A}(X)$ on $H_n$ defined by evaluating $m_A$ on the following exponential function:   $Y\in H_n \mapsto \exp(\ii \tr(X\,Y))\in \C$.
Hence we obtain : 
  \[\widehat{m_A}(X) = \int_{\U(n)} \exp(\ii \tr(X u^\star A u)) du \]
 As this quantity only depends on the respective eigenvalues of $X$ and $A$, \ie on the diagonal matrices $x=(x_1, x_2, \ldots x_n)$, and $\alpha = (\alpha_1, \alpha_2, \ldots, \alpha_n)$, it is then standard to rename the previous Fourier transform and consider the following two-variable function, called the Harish-Chandra orbital function: 
\be \label{HCfn}
{\CH}(\alpha, \ii x) =  \int_{\U(n)} \exp(\ii \tr (x u^\star \alpha u)) \, du
\ee

\subsubsection{The HCIZ integral} 
\label{secHCIZ}
The following explicit expression of  ${\CH}$ was found in \cite{HC,IZ}.
\be \label{HCIZ}
{\CH}(\alpha, \ii x) =
\Sf(n-1)\, \frac{(\text{det} \, e^ {\ii x_i \alpha_j})_{1 \leq i, j \leq n}}{\Delta(ix) \Delta(\alpha)}
\ee
where 
\[\Delta(x)=\Pi_{i<j} (x_i-x_j)\]
 is the Vandermonde determinant of the $x$'s.\\
Here and in the following we make use of the {\it superfactorial}
\be\label{superfact} \Sf(m):= \prod_{p=1}^m p!\,. \ee

{\ommit{
\subsubsection{Invariant probability measures on $H_n$ and their orbital transforms}
 An invariant measure on $H_n$ can be written $ f(A) \, dA$, where $f(A)$, its probability density (PDF), is a normalized, positive, and invariant  (\ie $f(u A u^\star) = f(A)$) function on $H_n$ and where $dA$ is the Lebesgue measure on $H_n$. It can be thought of as a measure on the space of $\U(n)$ orbits.  Now that we have appropriate analogs of the Dirac measures and of  their Fourier transforms, we may consider the orbital transform of an invariant measure:
 \[ \hat f (X) = \int_{H_n} \widehat{m_A}(X) \,  f(A) \, dA \,. \]
}}

\subsubsection{Convolution product of orbital measures} 
\label{convorb}
Take two orbits of the group $\U(n)$ acting on $H_n$, labelled by Hermitian matrices $A$ and $B$, and consider the corresponding orbital measures $m_A$, $m_B$.
The convolution product of the latter is defined as usual: with $f$, a function on $H_n$, one sets
  \[ <m_A \star m_B, f> = <m_A \otimes m_B, \blacktriangle(f)>\] 
 where 
 \[\blacktriangle(f)(a,b)\,  := f(a+b)\,.\] 
 This orbital analog of $\delta_{a+b}(c)$ has a non discrete support:  for $A, B \in H_n$, the support of 
 $m_{A,B}=m_A \star m_B$ is the set of $uAu^\star + vBv^\star$ for $u, v \in \U(n)$.
The probability density function $p$ of  $m_{A,B}$
is obtained by applying an inverse Fourier transformation to the product of Fourier transforms (calculated using $\widehat{m_A}(X)$) of the two measures:
\be\label{genform}
p(\gamma\vert\alpha, \beta) =  \frac{1}{(2\pi)^n} \, { \(\frac{\Delta(\gamma)}{\Sf(n) }\)^2} 
 \int_{\R^n} d^nx \, \Delta(x)^2 \, {\CH}(\alpha, \ii x) {\CH}(\beta, \ii x) {\CH}(\gamma, \ii x)^\star \,.
\ee
Notice that $p$ involves three copies of the HCIZ integral and that we wrote it as an integral on $\R^n$, whence the prefactor coming from the Jacobian of the change of variables.
We shall see below (formulae extracted from \cite{JB-I}) how to obtain quite explicit formulae for this expression.

\subsection{On polytopes}
\label{polytopes}
In the present context of  orbit sums and 
representation theory, one encounters two kinds of polytopes, not to be confused
with one another. \\
 On the one hand,  
given  two 
 multiplets $\alpha$ and $\beta$, ordered  as in (\ref{orderalpha}), 
 we have what may be called the {\it Horn polytope} $\bH_{\alpha \beta}$, which is the convex 
 hull of all possible ordered $\gamma$'s that appear 
 in the sum of the two orbits $\CO_\alpha$ and $\CO_\beta$. 
As proved by Knutson and Tao \cite{KT99} that Horn  polytope is identical to the convex set of real  
solutions to Horn's inequalities, including the inequalities (\ref{orderalpha}), applied to $\gamma$. 
For $\SU(n)$, this Horn  
polytope is $(n-1)$-dimensional.
\\

On the other hand, combinatorial models associate to such a triple 
$(\alpha,\beta;\gamma)$, 
  with $\gamma\in \bH_{\alpha \beta}$,
a family of
graphical objects that we call generically {\it pictographs}.  
This family depends on a number $(n-1)(n-2)/2$ of real parameters, subject to linear inequalities, 
thus defining a $d$-dimensional polytope $\widetilde{\CH}_{\alpha\,\beta}^\gamma$, with $d\le (n-1)(n-2)/2$.

These two types of polytopes are particularly useful in the discussion of highest weight representations 
of $\SU(n)$ and their tensor product decompositions.

Given two highest weight representations $V_\lambda$ and $V_\mu$ of $\SU(n)$,
we look at  the decomposition into irreps of $V_\lambda\otimes V_\mu$, or of $\lambda\otimes \mu$, in short,
see below sec.\,\ref{HornToLR}.
Consider a particular space of intertwiners {(equivariant morphisms)}
associated with a certain ``branching", \ie a  particular  term $\nu$ in 
that decomposition,  that we call an {\it admissible} triple $(\lambda,\mu ; \nu)$, see below Definition \ref{defadm}. 
 {Such $\nu$'s lie in the
 {\it tensor} polytope $\mathbf{H}_{\lambda\mu}$ inside the weight space.
 The multiplicity $N_{\lambda \mu}^\nu$ of
 $\nu$ in the tensor product $\lambda \otimes \mu$} is the dimension 
 of the space of intertwiners
 determined by the admissible triple $(\lambda,\mu; \nu)$.
{As proved in \cite{KT99}, is is also the number of pictographs with integral parameters. 
It is thus also the number of integral points in the second polytope that we now denote $\CH_{\lambda\,\mu}^\nu$.
These integral points  may be conveniently thought of as describing 
the different ``couplings'' of the three chosen irreducible representations.}

Pictographs are of several kinds.  All of them have three ``sides'' but one may distinguish two families: first we have those pictographs with sides labelled by integer partitions (KT-honeycombs \cite{KT99}, KT-hives \cite{KTW01}), then we have those pictographs with sides labelled by highest weight components of the chosen irreps (BZ-triangles \cite{BZ}, O-blades \cite{Ocn}, isometric honeycombs\footnote{The reader may look at \cite{CZ14} for an explicit descriptions and a few examples of O-blades and isometric honeycombs in the framework of the Lie group SU(3).  See also our  $\SU(4)$ example in sec.\,\ref{exampleSU4}.}). {For convenience, we refer to $\CH_{\lambda\,\mu}^\nu$ as the ``hive polytope'',
or also   ``the polytope of honeycombs''.}

As mentioned above, for $\SU(n)$,
 and for an admissible triple $(\lambda,\mu ; \nu)$, the dimension of the hive polytope is
 $(n-1)(n-2)/2$: this  may  be taken as a definition of a ``generic triple", but  
 see below Lemma \ref{conj1} for  a more precise characterization.
 The cartesian equations for the boundary hyperplanes have 
 integral coefficients, the hive polytope  is therefore a rational polytope.
All the hive polytopes that we consider in this article are ``integral hive polytopes'' in the terminology of \cite{KTT04},  however the corners of all such polytopes (usually called ``vertices'') are not always integral points, therefore an ``integral hive polytope'' is not necessarily an integral polytope in the usual sense: 
 the convex hull of its integral points is itself a polytope, but there are cases where the latter is strictly included in the former. We shall see an example of this situation in sec.\,\ref{exampleSU5}.

 We shall return later to these polytopes and to the counting functions of their integral points, in relation with stretched Littlewood-Richardson coefficients, see sec.\,\ref{LREpols}.

\subsection{Some formulae and results from paper \cite{JB-I}}

\subsubsection{Determination of the density  $p$ and of the kernel function ${\CJ}_n$}
\label{pandIn}
	Some general expressions for the three variable function $p$ were obtained in 	\cite{JB-I}.
	 For the convenience of the reader, 	 we repeat them here. 	 

The determinant entering the HCIZ integral is written as

\bea\label{detx} \det e^{\ii x_i \alpha_j }&=&e^{\ii \inv{n}\sum_{j=1}^n x_j \sum_{k=1}^n \alpha_k} \det e^{\ii (x_i-\inv{n}\sum x_k) \alpha_j }
\\ \label{detSUn} &=& e^{\ii \inv{n}\sum_{j=1}^n x_j\sum_{k=1}^n \alpha_k}   \sum_{P\in S_n} \varepsilon_P 
\prod_{j=1}^{n-1} e^{\ii (x_j-x_{j+1})  ( \sum_{k=1}^j \alpha_{P(k)} -\frac{j}{n} \sum_{k=1}^n \alpha_k)}\,,\eea
where $\varepsilon_P$ is the signature of permutation $P$.

In the product of the three determinants  entering (\ref{genform}), the prefactor $e^{\ii \sum_{j=1}^n x_j \sum_{k=1}^n (\alpha_k+\beta_k-\gamma_k)/n}$
yields, upon integration over $\inv{n}\sum x_j$, $2\pi$ times a Dirac delta of $\sum_k(\alpha_k+\beta_k-\gamma_k)$, expressing the conservation of the trace
in Horn's problem. One is left with an expression involving an integration over $(n-1)$ variables
$u_j:=x_j-x_{j+1}$.

\bea\label{pSUn} p(\gamma|\alpha,\beta)&=&
 \frac{\Sf(n-1)}{ n!}\delta(\sum_k(\alpha_k+\beta_k-\gamma_k))\, \frac{ \Delta(\gamma)}{\Delta(\alpha)\Delta(\beta)}  \,\CJ_n(\alpha,\beta;\gamma) 
\\ \label{In}
\CJ_n(\alpha,\beta;\gamma) &=&\,  \frac{\ii^{-n(n-1)/2}}{2^{n-1}\, n! \,\pi^{n-1}}  
\sum_{P,P',P''\in S_n}\varepsilon_P\,\varepsilon_{P'}\, \varepsilon_{P''}\, 
\int_{\R^{n-1}} \frac{d^{n-1}u}{\widetilde\Delta(u)}\,\prod_{j=1}^{n-1}  e^{\ii u_j A_j(P,P',P'')}
\\ \label{Aj} A_j(P,P',P'')&=& \sum_{k=1}^j (\alpha_{P(k)}+\beta_{P'(k)}-\gamma_{P''(k)}) -
\frac{j}{n} \sum_{k=1}^n (\alpha_k+\beta_k-\gamma_k),
\eea
 where the Vandermonde $\Delta(x)$ has been rewritten as 
 \be\label{tildeDelta}\widetilde\Delta(u):=
  \prod_{1\le i < j\le n} (u_i+\cdots + u_{j-1})\,. \ee 
\\

\subsubsection{Discussion}
\label{discIn}
 Several properties of $p(\gamma\vert\alpha, \beta)$ and of ${\CJ}_n$ are described in the  paper \cite{JB-I}.	
We only summarize here the information that will be relevant for our discussion relating these functions to the Littlewood-Richardson multiplicity problem.

Note that the above expression of $A_j$ is invariant under simultaneous translations of all 
$\gamma$'s
$$\forall\ i\qquad \gamma_i\to \gamma_i+c\,\qquad\quad   c\in \R\,.$$
In the original Horn problem, this reflects the fact that
 the PDF $p(\gamma|\alpha,\beta)$ of eigenvalues of $C=A+B$ is the same as that of $C+c I$, with a shifted support.  
Therefore in the computation of $\CJ_n(\alpha,\beta;\gamma)$,  one has a freedom in the choice of a   ``gauge"\\[4pt]
(a) either $\gamma_n=0$, \\[3pt]
(b) or $\gamma$ such that 
\be\label{equtr} \sum_i \gamma_i = \sum_i (\alpha_i +\beta_i)\,,\ee
(c) or any other choice,\\
provided one takes into account the second term in the rhs of (\ref{Aj}) (which vanishes in case (b)).\\
Note also that  enforcing (\ref{equtr}) starting from an arbitrary $\hat\gamma$ implies to translate
$\hat \gamma\to \gamma=\hat\gamma+c$, with  
$ c=\inv{n} (\sum_i \alpha_i +\sum_i \beta_i -\sum_i \hat\gamma_i)$. If the original $\hat\gamma$
has integral components, this is generally not the case for the final $\gamma$.
\\[10pt]
$\CJ_n(\alpha,\beta;\gamma)$ has the following properties that will be used below: 
\\
 -- (i) As apparent on (\ref{In}), it is an antisymmetric function of $\alpha$, $\beta$ or $\gamma$ under the action of
the Weyl group of $\SU(n)$ (the symmetric group $S_n$). As already said,  we choose throughout this paper
the ordering (\ref{orderalpha})
 and  likewise for $\beta$ and $\gamma$. \\
For  $(\alpha,\beta;\gamma)$ satisfying (\ref{equtr})\\
-- (ii) $\CJ_n(\alpha,\beta;\gamma)$ is  piecewise polynomial, homogeneous of degree $\oh(n-1)(n-2)$ in $\alpha, \beta, \gamma$ in the generic case;\\
-- (iii)  as a function of $\gamma$,  it is of class $C^{n-3}$. This follows by the Riemann--Lebesgue 
theorem from the decay at large $u$ of the integrand in (\ref{In}), see \cite{JB-I};\\
 -- (iv) it is non negative inside the  polytope $\bH_{\alpha\beta}$, cf sec.\,1.2; \\
-- (v) it vanishes for ordered $\gamma$ outside  $\bH_{\alpha \beta}$;\\
-- (vi) by continuity (for $n\ge 3$) it vanishes for $\gamma$ at the boundary of $\bH_{\alpha \beta}$;\\
-- (vii) it also vanishes whenever
 at least two components of $\alpha$ or of $\beta$ coincide\footnote{If $\alpha$ and $\beta$ are Young partitions describing the highest weights $\lambda$, $\mu$ of two $\U(n)$ or $\SU(n)$ irreps, this occurs  
 when some Dynkin label of $\lambda$ or $\mu$ vanishes, \ie when $\lambda$ or $\mu$ belongs to a wall of the 
 dominant Weyl
 chamber $C$.}: this follows from 
the antisymmetry mentionned above;  \\
-- (viii) its normalization follows from that of the probability density $p$, (normalized of course by\\
$\int_{\R^n} d^n\gamma\, p(\gamma|\alpha,\beta)=1$), hence
\be\label{normIn} 
\int_{\bH_{\alpha\beta}} d^{n-1}\gamma\,
\frac{ \Delta(\gamma)}{\Delta(\alpha)\Delta(\beta)}  \,\CJ_n(\alpha,\beta;\gamma)  
= \inv{\Sf(n-1)}
\ee
 which equals
 $1, \inv{2}, \inv{12}, \inv{288}, \inv{34560},
 \cdots$ for $n=2,3,4,\cdots$. \\
 
\smallskip
 As mentioned above, it is natural to adopt the following definition
 \begin{definition}\label{defgen} A triple $(\alpha,\beta;\gamma)$ is called {\rm generic} if 
 $\CJ_n(\alpha,\beta;\gamma)$ is non vanishing. \end{definition} 
 
By a slight abuse of language, when dealing with triples of highest weights $(\lambda, \mu;\nu)$,
we say that such an admissible triple is generic iff the associated triple $(\alpha,\beta;\gamma)$ is,
see below sec. \ref{part-hw}. By another abuse of language, we also refer to a {\it single} highest weight $\lambda$ as
{\it generic} iff none of its Dynkin indices vanishes, \ie iff $\lambda$ does not lie on one of the walls
of the dominant Weyl chamber, or if equivalently the associated $\alpha$ has no pair of equal components.

From its interpretation as a probability density (up to positive factors), it is clear that $\CJ_n$ could vanish 
 at most on subsets of measure zero inside the Horn (or tensor) polytope. 
 Actually it does not vanish besides the cases mentioned in points (v-vii) of the previous list, as we now argue.

We want to  construct the linear span of honeycombs $\widetilde{\CH}_{\alpha\beta}^\gamma$ defined above in sect. 1.2.
 We first consider what may be called the ``$\SU(n)$ case", where $\alpha_n=\beta_n=0$ and $\gamma_n$ is fixed by (12).  
By relaxing   the inequalities on the $(n-1)(n-2)/2$ parameters defining the usual honeycombs,  one builds
a vector space of dimension $\frac{1}{2} (n-1) (n+4) =3(n-1) +(n-1)(n-2)/2$ whose elements are sometimes called  {\sl real} honeycombs. One may construct a basis of ``fundamental honeycombs", see [10],  and 
consider arbitrary linear combinations, with real coefficients, of these basis vectors.
The components of any admissible triple$(\alpha,\beta,\gamma)$, 
depend linearly of the components of the associated honeycombs along the chosen basis. 
In such a way, one obtains a surjective linear map,  from the vector space of real honeycombs,  to the vector space $\R^{3(n-1)}$. 

One sees immediately that its fibers are affine spaces of dimension $d_{{\rm max}}=(n-1)(n-2)/2$, 
and for fixed  $\alpha,\beta$ they are indexed by $\gamma$, \ie by points of $\R^{(n-1)}$.
By  taking into account 
the inequalities defining usual honeycombs,  but still working with real coefficients, the fibers of this map restrict to compact polytopes whose affine dimension $d$ is at most equal to  $d_{{\rm max}}$ (the dimension can be smaller, because of the inequalities that define bounding hyperplanes). 
For given $\alpha$ and $\beta$, if $\gamma$ belongs to  
the Horn polytope $\widetilde{\mathbf{H}}_{\alpha\beta} \subset \R^{(n-1)}$, the corresponding restricted fiber is nothing else than the associated hive polytope $\widetilde{\CH}_{\alpha\beta}^\gamma$. We therefore obtain a map $\pi$ whose target set is the Horn polytope, a convex set, and whose fibers are compact polytopes.
We then make use of  the following result\footnote{We thank Allen Knutson for pointing this out to us.}:  the dimension of the fibers of $\pi$ is constant on the interiors of the faces of its target set.
In particular, it is constant on the interior of its face of codimension $0$, which is the interior of 
the Horn polytope $\widetilde{\mathbf{H}}_{\alpha\beta} $.

In the present situation this tells us that the dimension of  $\pi^{-1}(\gamma) \,=\,\widetilde{\CH}_{\alpha\beta}^\gamma$
which is the fiber above  $\gamma$, is constant 
when $\gamma$  belongs to the interior of the Horn polytope $\widetilde{\mathbf{H}}_{\alpha\beta} $.
 In particular, its $d$-dimensional volume, where $d$ has its maximal value $d=(n-1)(n-2)/2$ for $\SU(n)$, cannot vanish there.
We shall see later (in section 3) 
that this volume is given by  $\CJ_n(\alpha,\beta;\gamma)$.

In the case of GL$(n)$, (with $\alpha_n,\beta_n$ non fixed to 0),  the argument is similar,  so we have: 

 \begin{lemma}\label{conj1}{For $\alpha$ and $\beta$ with distinct components,  
the function $\CJ_n(\alpha,\beta;\gamma)$ does not vanish for $\gamma$ inside the polytope 
 $\bH_{\alpha \beta}$.}\end{lemma}


\section{From Horn to Littlewood-Richardson and from orbital transforms to characters}
\label{HornToLR}
\subsection{Young partitions and highest weights}
\label{part-hw}

An irreducible  polynomial representation
 of GL$(n)$ or an irrep of $\SU(n)$, denoted
 $V_\lambda$,  is
characterized by its highest weight  $\lambda$ (h.w. for short). One may use alternative notations, describing 
this highest weight either by its Dynkin indices (components in a basis of fundamental weights)
 $\lambda_i$, $i=1,\cdots,n$,  and $\lambda_n=0$ in $\SU(n)$ ; or by its Young components,
 \ie the lengths of rows of the corresponding Young diagram: $\alpha=\ell(\lambda)$, \ie 
\be \label{defell} \ell_i(\lambda)=\sum_{j=i}^{n} \lambda_j\, \qquad i=1,\cdots,n\, .\ee 
Note that such an $\alpha=\ell(\lambda)$ satisfies the ordering condition (\ref{orderalpha}).\\
  In the decomposition into irreps of the tensor product of two such irreps
 $V_\lambda$ and $V_\mu$ of GL$(n)$,   we denote by  $ N_{\lambda\mu}^\nu $  
the Littlewood-Richardson (LR)  multiplicity of $V_\nu$.
\\
 As recalled above, 
$ N_{\lambda\mu}^\nu $  equals the number of honeycombs with integral labels 
and boundary conditions $\alpha=\ell(\lambda),\ \beta=\ell(\mu), \ \gamma=\ell(\nu)$,
 \ie the number of integral points in the polytope $\CH_{\lambda\,\mu}^\nu$  \cite{KT99}.
 \\ 
Given three U($n$) (resp. $\SU(n)$) weights $\lambda,\mu,\nu$, for instance described by their $n$ (resp. $n-1$) components along the basis of fundamental weights, 
invariance under the U($1$)  center of U($n$) (resp. the $\Z_n$  center of $\SU(n)$), 
tells us that a necessary condition for the non-vanishing of $N_{\lambda\mu}^\nu$
is  $\sum_{j=1}^{n}  j(\lambda_j+\mu_j-\nu_j) =0$ (resp. $\sum_{j=1}^{n-1}  j(\lambda_j+\mu_j-\nu_j) =0 \mod n$).
\\
Given three $\SU(n)$ weights $\lambda,\mu,\nu$ obeying the above $\SU(n)$ condition, one can build 
three U($n$) weights (still denoted  $\lambda,\mu,\nu$) obeying the U($n$) condition
by setting $\lambda_n=\mu_n=0$ and $\nu_n=\inv{n}\sum_{j=1}^{n-1}  j(\lambda_j+\mu_j-\nu_j)$;
in terms of partitions, with $\alpha = \ell(\lambda)$, $\beta = \ell (\mu)$ and $\gamma = \ell (\nu)$, the obtained triple $(\alpha, \beta; \gamma)$
 automatically obeys eq.~(\ref{equtr}).\\
More generally we shall refer to a U($n$) triple such that the equivalent U($n$) conditions eq.~(\ref{equtr}), or eq.~(\ref{adm}) below, hold true, as a U($n$)-compatible triple, or  a {\it compatible} triple, for short.
 \begin{definition}\label{defcomp}{A triple $(\lambda,\mu;\nu)$ of $\U(n)$ weights is said to be {\rm compatible} iff 
\be\label{adm}{\sum_{k=1}^{{n}} k (\lambda_k+\mu_k -\nu_k)} =0\,.\ee }
\end{definition}
For triples of $\SU(n)$ weights, we could use the same terminology, weakening the above condition (\ref{adm}) since it is then only assumed to hold modulo $n$, 
but in the following we shall always extend such $\SU(n)$-compatible triples to U($n$)-compatible triples, as was explained previously.
\\
We also recall another more traditional definition
 \begin{definition}\label{defadm}{A triple $(\lambda,\mu;\nu)$ of $\U(n)$ or $\SU(n)$ weights is said to be {\rm admissible} iff 
$N_{\lambda\mu}^\nu\ne 0$\,.}
\end{definition}
The reader should remember (at least in the context of this article\,!) the difference between compatibility and admissibility, the former being obviously a necessary condition for the latter.


\bigskip

For given $\lambda$ and $\mu$, or equivalently, given $\alpha$ and $\beta$, if $N_{\lambda\mu}^\nu\ne 0$ 
for some h.w. $\nu$, the corresponding $\gamma$ must lie 
inside {\it or on the boundary of} the Horn polytope $\bH_{\alpha\beta}$, by definition of the latter. 
Since for $n\ge 3$ the function 
$\CJ_n(\alpha,\beta;\gamma)$ is continuous and vanishes on the boundary of its support,  evaluating it for 
$\alpha,\beta, \gamma$  
does not provide a strong enough criterion to identify  admissible triples $(\alpha,\beta;\gamma)$.

\subsection{Relation between  Weyl's character formula and the HCIZ integral}
\label{WHCIZ}

There is an obvious  similarity between the general form (\ref{genform}) of the PDF 
$p(\gamma|\alpha,\beta)$ and the expression of the LR multiplicity 
$N_{\lambda\mu}^\nu$ as the integral of the product of characters $\chi_\lambda \chi_\mu \chi_\nu^*$ over the unitary group  SU(n) or over its Cartan torus $\T_n \,= \U(1)^{n-1}$ 
\be\label{LRintSUn}
N_{\lambda\mu}^\nu =\int_{\SU(n)}du\, \chi_\lambda(u) \chi_\mu(u) \chi_\nu^*(u)  \quad  \text{or} \quad
N_{\lambda\mu}^\nu =\int_{ \T_n} dT\, \chi_\lambda(T) \chi_\mu(T) \chi_\nu^*(T) \ee
with the normalized Haar measure on $ \T_n$,
\be\label{HaarTn} dT=  \inv{ (2\pi)^{n-1} n! }|\Delta(e^{\ii t})|^2  \prod_{i=1}^{n-1} dt_i\,,\ee
  for
\be\label{diagT} T=\diag(e^{\ii t_j})_{j=1,\cdots,n} \qquad \mathrm{with}\ \sum_{j=1}^n t_j=0\,.\ee

This similarity finds its root in the 
Kirillov \cite{Kirillov} formula expressing $\chi_\lambda$ as the orbital
function ${\CH}$ relative to $\CO_{\ell(\lambda+\rho)}$, defined in (\ref{HCfn}), see below 
(\ref{char2}-\ref{Weylform});  note the shift of $\lambda$ by the Weyl vector $\rho$, the half-sum of positive roots.

Recall Weyl's 
formula for the dimension of  the  vector space $V_\lambda$ of h.w. $\lambda$
 \be\label{dimVla}   \dim V_\lambda = \frac{\Delta(\alpha')}{\Sf(n-1)}\qquad \mathrm{with}\ \alpha'=\ell(\lambda+\rho)\,, \quad \text{  and $\ell$ as defined  in (\ref{defell})}\,.
 \ee
From a geometrical point of view, this formula expresses $ \dim V_\lambda $ as the volume of a group orbit normalized by the volume of $\SU(n)$, the latter being also equal to $\Sf(n-1)$, once
a  natural Haar measure has been chosen, see \cite{Macdonald}.
 \\

\subsubsection{From group characters to Harish-Chandra orbital functions}

Kirillov's formula \cite{Kirillov} relates Weyl's $\SU(n)$ character formula with the orbital function of $\CO_{\alpha'} $.
Here and below, the prime on $\alpha'$ refers to the value of $\alpha$, 
for the {\it shifted}  highest weights $\lambda+\rho$ 
\be \alpha'=\ell(\lambda+\rho)\,,\ee and likewise for $\beta', \gamma'$. 
Indeed evaluated on an element $T$ of the $\SU(n)$ Cartan torus  as in (\ref{diagT}),
Weyl's character formula reads 
\be\label{char1}  \chi_{\lambda} (T):= \tr_{V_\lambda}(T)=\frac{\det{e^{\ii t_i \alpha'_j}}} {\Delta(e^{\ii t})} \quad \text{with} \quad \Delta(e^{\ii t}) = {\prod_{1\le i< j\le n} (e^{\ii t_i}-e^{\ii t_j})} \,, 
\ee
or in terms of the orbital function $\CH$ defined in (\ref{HCfn}) and made explicit in (\ref{HCIZ})
\be\label{char2}  \chi_{\lambda} (T)=  \frac{ \Delta(\alpha')
}{\Sf(n-1)}\ \(\prod_{1\le i< j\le n} \frac{\ii(t_i-t_j)} {(e^{\ii t_i}-e^{\ii t_j})} \)\CH(\alpha', \ii t)  
\ee
or, owing to  the Weyl dimension formula (\ref{dimVla})  
\be\label{Weylform}  \frac{\chi_{\lambda} (T)}{\dim V_\lambda} = 
\frac{\Delta({\ii t})}{\Delta( e^{\ii t})}\,\CH(\alpha', \ii t)\, .\ee


\subsubsection{The polynomial $R_n(T)$}
\label{polRn}
  
 Consider the following (semi-convergent) integral
  $$ J=\int_\R du \frac{e^{\ii u A}}{u}\qquad A\in \R$$
  a one-dimensional analogue of the integral encountered in (\ref{In}). If $A$ is a 
  half-integer, we may write
 \bea\nonumber
A\ \hbox{half-integer} \qquad J&=& \int_{-\pi}^\pi du\, e^{\ii u A} \sum_{n=-\infty}^\infty \frac{(-1)^n}{u+n(2\pi)}
 \\ \nonumber
 &=& \int_{-\pi}^\pi du\, e^{\ii u A} \inv{2\sin(u/2)}\eea
  according to a well-known identity.  If $A$ is an integer, 
  the previous sum over $n$ is understood as a principal value.
  Then
  \be\nonumber  A\ \mathrm{integer} \qquad J = \int_{-\pi}^\pi du\, e^{\ii u A}\  \ P.V.\!\!\! \sum_{n=-\infty}^\infty \inv{u+n(2\pi)}
 = \int_{-\pi}^\pi du\, e^{\ii u A} \inv{2\tan(u/2)}\ee
  
  We now repeat this simple calculation for the $(n-1)$-dimensional integral
  appearing in (\ref{In}), evaluated  either for unshifted $\alpha, \beta, \gamma$ or 
  for shifted $\alpha', \beta', \gamma'$, 
    associated  as above with  a compatible triple of highest 
  weights $(\lambda,\mu;\nu)$.
  
  First we observe that the determinant $ \det e^{\ii (x_i-\inv{n}\sum x_k) \alpha'_j }$ that appears 
  in the first line of (\ref{detSUn}) is nothing else than the numerator of Weyl's formula (\ref{char1}) 
  for the $\SU(n)$ character 
  $\chi_\lambda(T)$, evaluated for the unitary and unimodular matrix 
  \be\label{T}T=\diag\big(e^{\ii (x_i-\inv{n}\sum x_k)}\big)\,.\ee
Henceforth we take $t_i=(x_i-\inv{n}\sum x_k)$, $\sum t_i=0$. 
Consider now the product of three such determinants as they appear in the 
computation of $\CJ_n(\alpha',\beta';\gamma')$, see (\ref{In}).
 Each factor $e^{\ii \sum_j u_j A_j }$, 
  under $2\pi$-shifts of the variables $u_j:= t_j-t_{j+1}$,
  $u_j\to u_j + p_j (2\pi)$,  is not necessarily periodic, because of the second term of $A_j$ in (\ref{Aj}):
  $$ e^{\ii \sum_j u_j A_j } \to e^{\ii  \sum_j u_j A_j }  e^{- 2\pi \ii \sum_j   \frac{j p_j}{n}\sum_k  (\alpha'_k
  +\beta'_k-\gamma'_k)}\,.$$
  Indeed, for $\alpha'=\ell(\lambda+\rho)$, etc,  we have 
    $$\sum_{k=1}^n (\alpha'_k +\beta'_k-\gamma'_k)=
  {\sum_{k=1}^{n-1} k (\lambda_k+\mu_k -\nu_k)}
   +\frac{n(n-1)}{2}\,,$$ 
   the first term of which vanishes  for a compatible  triple $(\lambda,\mu;\nu)$, 
   see (\ref{adm}). 
    Thus we find that under the above shift, $ e^{\ii \sum_j u_j A_j } \to e^{\ii  \sum_j u_j A_j } (-1)^{\sum_j j (n-1) p_j}$.
  For $n$ odd, like in SU(3), the numerator is $2\pi$-periodic in each variable $u_j$. For $n$ even, however, 
  we have a sign $(-1)^{j p_j}$.   
  We may thus compactify the integration domain  of the $u$-variables, 
  bringing it from $\R^{n-1}$ back 
  to $(-\pi,\pi)^{n-1}$ by translations $u_j\to u_j +(2\pi) p_j$, while taking the above sign into account.
  Thus for a compatible triple $(\lambda,\mu;\nu)$ and the $A_j$'s standing for the expressions of (\ref{Aj})
  computed at shifted weights $\alpha'=\ell(\lambda+\rho)$ and likewise for $\beta'$ and $\gamma'$, we have

 $$
\int_{\R^{n-1} }  \frac{\prod_{j=1}^{n-1} d u_j \, e^{\ii u_j A_j}}{\widetilde \Delta(u)}
=  \int_{(-\pi,\pi)^{n-1} }  \prod_{j=1}^{n-1} d u_j \, e^{\ii u_j A_j}  D_n$$
where
\be\label{defDn}
D_n= \sum_{p_1, \cdots, p_{n-1} =-\infty}^\infty (-1)^{\sum_j j p_j (n-1)} \prod_{1\le i < i'\le n}   \frac{1}{u_i+u_{i+1}+\cdots +u_{i'-1}+(p_i+\cdots +p_{i'-1})(2\pi)}\, ,\ee
 a sum that always converges.
Now define
\bea\label{defpin} 
\varpi_n&:=& \prod_{1\le i < i'\le n} 2 \sin(\oh(u_i+u_{i+1}+\cdots +u_{i'-1}))=
{ \ii^{-n(n-1)/2}} \Delta(e^{\ii t_i}) \\
 \label{shiftDelta}   R_n(T)
 &:=& D_n \varpi_n\,.
    \eea
$R_n$, as  defined by (\ref{shiftDelta}), is a function of $T$ with no singularity, 
since all the poles of the original
  expression  $\widetilde \Delta(u)^{-1}$ have been embodied in the denominator ${\Delta(e^{\ii t_i})}$. It 
  must be a polynomial  in $T$ and $T^\star$,  invariant under permutations 
and complex conjugation, hence a real symmetric polynomial  of the $e^{\ii t_j}$.
(Since $\det T=1$, $T^\star$ is itself a polynomial in $T$.) 
We conclude that $R_n(T)$ may be expanded on real characters $\chi_\kappa(T)$, $\kappa\in \CK$,  with $\CK$ a finite
 $n$-dependent set of highest weights. Moreover
  $R_n(I)=1$,  as may be seen by looking at the small $t$ limit of (\ref{shiftDelta}). 
 Thus
 \begin{proposition}\label{propcomp}{The integrals over $\R^{n-1}$ appearing in $\CJ_n(\alpha',\beta';\gamma')$ in (\ref{In}),
  for $\alpha'=\ell(\lambda+\rho),\, \beta'=\ell(\mu+\rho),\, \gamma'=\ell(\nu+\rho)$,  $(\lambda,\mu;\nu)$ a compatible triple, may be
 ``compactified" in the form
 \be\label{compact}\int_{\R^{n-1} }  \frac{\prod_{j=1}^{n-1} d u_j  e^{\ii u_j A_j}}{\widetilde \Delta(u)}
 ={ \ii^{n(n-1)/2}}  \int_{(-\pi,\pi)^{n-1} }  \prod_{j=1}^{n-1} d u_j \,  e^{\ii u_j A_j} \, \dfrac{R_n(T)}{\Delta(e^{\ii t_i})} \ee
 where the real polynomial $R_n(T)$ is defined through  (\ref{shiftDelta}). 
There exists a finite, $n$-dependent set $\CK$ of highest weights such that $R_n(T)$ may be written as  
 a  linear combination
 $R_n(T)=\sum_{\kappa \in \CK}
 r_\kappa \chi_\kappa(T)$ of real characters. The coefficients $r_\kappa$ are rational and such that, 
when evaluated at the identity matrix, $R_n(I)=1$. } \end{proposition}

Consider now the similar computation,  again for a compatible triple $(\lambda,\mu;\nu)$ 
 but with the $A_j$'s standing for the expressions of (\ref{Aj})
  computed at {\it unshifted weights}, \ie with $\alpha=\ell(\lambda)$ and likewise for $\beta$ and $\gamma$. 
 If the triple $(\alpha,\beta;\gamma)$ is non generic, $\CJ_n(\alpha,\beta;\gamma)=0$. If it is generic,
  and $n$ is odd, $(\alpha,\beta;\gamma)$ may be thought of as associated with the shift of the 
  compatible triple $(\lambda-\rho,\mu-\rho\,;\nu-\rho)$. Thus for $n$ odd, this new calculation yields 
  the same result as above.
  For $n$ even, however,   the latter triple is no longer compatible and
  a separate calculation has to be carried out. It is easy to see that the same line of reasoning leads to a 
  modification of the formula (\ref{shiftDelta}) and to a new family of real symmetric polynomials $ {\widehat R}_n(T)$,
  according to

\bea \nonumber
 && \int_{\R^{n-1} }  \frac{\prod_{j=1}^{n-1} d u_j  e^{\ii u_j A_j}}{\widetilde \Delta(u)}
=  \int_{(-\pi,\pi)^{n-1} }  \prod_{j=1}^{n-1} d u_j  e^{\ii u_j A_j} {\widehat D}_n
\\ \label{defhatD}
{\widehat D}_n &:=& \sum_{p_1, \cdots, p_{n-1} =-\infty}^\infty \, 
\prod_{1\le i < i'\le n} \,  \frac{1}{u_i+u_{i+1}+\cdots +u_{i'-1}+(p_i+\cdots +p_{i'-1})(2\pi)}\qquad
\\ \label{shiftDeltaev}   {\widehat R}_n(T)&:=& {\widehat D}_n \varpi_n\,,
  \eea
 with the same $\varpi_n$ as in (\ref{defpin}).
Note that the sum in (\ref{defhatD}) is convergent for $n>2$. The case $n=2$ requires a 
 special treatment, see below in sec. \ref{R2-R3}. 
 
  \begin{proposition}\label{propcompev}{The integrals over $\R^{n-1}$ appearing in $\CJ_n(\alpha,\beta;\gamma)$ in (\ref{In}), 
  for $\alpha=\ell(\lambda),\, \beta=\ell(\mu),\, \gamma=\ell(\nu)$,  $(\lambda,\mu;\nu)$ a compatible triple, may be
 compactified in the form
 \be\label{compact}\int_{\R^{n-1} }  \frac{\prod_{j=1}^{n-1} d u_j  e^{\ii u_j A_j}}{\widetilde \Delta(u)}
 ={ \ii^{n(n-1)/2}}  \int_{(-\pi,\pi)^{n-1} }  \prod_{j=1}^{n-1} d u_j \,  e^{\ii u_j A_j} \, \dfrac{{\widehat R}_n(T)}{\Delta(e^{\ii t_i})} \ee
 where the real polynomial ${\widehat R}_n(T)$ is defined through  (\ref{shiftDeltaev}). 
There exists a finite $n$-dependent set $\widehat{\CK}$ of highest weights such that ${\widehat R}_n(T)$ may be written as  
 a  linear combination
 ${\widehat R}_n(T)=\sum_{\kappa \in \widehat{\CK}}
 \hat r_\kappa \chi_\kappa(T)$ of real characters. The coefficients $\hat r_\kappa$ are rational and such that, 
when evaluated at the identity matrix, $\widehat{R}_n(I)=1$.
 For $n$ odd, 
the following objects coincide with those of Proposition~\ref{propcomp}: $\widehat{R}_n\equiv R_n$,  
$\widehat{\CK}=\CK$ and $r_\kappa=\hat r_\kappa$. } \end{proposition}

 A method of calculation and explicit expressions for low values of $n$
 of the polynomials $R_n$,  ${\widehat R}_n$ and of the sets $\CK$, $\widehat{\CK}$
 will be given in sections \ref{Pexpressions}  and \ref{Rnlown}, establishing the rationality of the coefficients $r_\kappa, \hat r_\kappa$.
We shall see that the polynomial $R_n$ is equal to $1$ for $n=2$ and $n=3$, but  non-trivial when $n \geq 4$.
In contrast, already for $n=2$, $\widehat{R}_2(T)=\oh \chi_{1}(T)$.
These expressions of $R_n$ and $\widehat{R}_n$ for low $n$ suggest the following conjecture

\begin{conjecture}\label{conj2}{
The coefficients $r_\kappa$  and $\hat r_\kappa$ are non negative.} \end{conjecture}
As we shall see below in sec.\,\ref{conseqTh1} (v), this Conjecture \ref{conj2} is related to Lemma \ref{conj1}.\\
\bigskip

\subsection{Relation between $\CJ_n$ and LR coefficients}
 We may now complete the computation of $\CJ_n(\alpha',\beta';\gamma')$ and $\CJ_n(\alpha,\beta;\gamma)$. We  rewrite
  $$ \inv{\Delta(e^{\ii t})} = |\Delta(e^{\ii t})|^2 \inv{\Delta(e^{\ii t}) \Delta(e^{\ii t}) \Delta(e^{\ii t})^*}\,,$$ 
   the first term $|\Delta(e^{\ii t})|^2$ 
  is what is needed for 
    writing the normalized Haar measure over the $\SU(n)$ Cartan torus $\T_n$, see (\ref{HaarTn}), 
 while the three Vandermonde determinants in the denominator provide the desired denominators of Weyl's character formula.

 Putting everything together  we find
 \begin{theorem}\label{CI-LR}
1. For a compatible  triple
   $(\lambda,\mu;\nu)$, the integral $\CJ_n$ of (\ref{pSUn}-\ref{In}), evaluated for the shifted 
  weights $\lambda+\rho$ etc, or for the corresponding $\alpha'=\ell(\lambda+\rho),\,\beta' = \ell(\mu+\rho),\,
  \gamma' = \ell(\nu+\rho)$, 
   may be recast as  
 \be\label{CI-LR0} 
 \CJ_n(\alpha',\beta'; \gamma')=\int_{ \T_n} dT\, \chi_\lambda(T) \chi_\mu(T) \chi_\nu^*(T) \,R_n(T) \ee
  where the integration is carried out on the Cartan torus with its  normalized Haar measure.
Writing $R_n(T)=\sum_{\kappa \in\CK}
 r_\kappa \chi_\kappa(T)$ as in Prop. \ref{propcomp}, this may be rewritten as 
 \bea\nonumber  \label{CI-LR1}  \CJ_n(\alpha',\beta'; \gamma') &=& {\sum_{\kappa\in \CK \atop \nu'} r_\kappa
 N_{\lambda\mu}^{\nu'}    N_{\kappa\nu}^{\nu'} }
 \\ \label{CI-LR1} &=&\sum_{\nu'} \hat c^{(\nu)} N_{\lambda\mu}^{\nu'}\eea
 where the sum runs over the finite set of  irreps $\nu'$ obtained in the decomposition of $\oplus_{\kappa\in \CK}
  (\nu\otimes  \kappa)$, with rational coefficients 
   $c_{\nu'}^{(\nu)} =  \sum_{\kappa\in\CK} N_{\kappa\nu}^{\nu'} \, r_\kappa$.\\
2. For a compatible  triple
   $(\lambda,\mu;\nu)$ of weights not on the boundary of the Weyl chamber, 
   the integral $\CJ_n$ of (\ref{pSUn}-\ref{In}), evaluated for the unshifted 
  weights $\lambda,\, \mu,\, \nu$, or for the corresponding $\alpha=\ell(\lambda),\,\beta = \ell(\mu),\,
  \gamma = \ell(\nu)$, 
   may be recast as  
 \be\label{CI-LR0ev} \CJ_n(\alpha,\beta; \gamma) 
 =\int_{ \T_n} dT\, \chi_{\lambda-\rho}(T) \chi_{\mu-\rho}(T) \chi_{\nu-\rho}^*(T) \,{\widehat R}_n(T) \ee 
  where the integration is carried out on the Cartan torus with its  normalized Haar measure.
Writing ${\widehat R}_n(T)=\sum_{\kappa \in \widehat{\CK}}
 \hat r_\kappa \chi_\kappa(T)$ as in Prop. \ref{propcompev}, this may be rewritten as 
 \bea \label{CI-LR1ev0}    \CJ_n(\alpha,\beta; \gamma)  &=& {\sum_{\kappa\in\widehat{\CK}\atop \nu'}\hat r_\kappa
  N_{\lambda-\rho\,\mu-\rho}^{\nu'} N_{\kappa\, \nu-\rho}^{\nu'}} \\
   \label{CI-LR1ev}  
&=& \sum_{\nu'} \hat c_{\nu'}^{(\nu)} N_{\lambda-\rho\,\mu-\rho}^{\nu'}\eea
 where the sum runs over the finite set of  irreps $\nu'$ obtained in the decomposition of $\oplus_{\kappa\in \widehat{\CK}}
 \big( (\nu-\rho) \otimes  \kappa\big)$, with rational coefficients $\hat c_{\nu'}^{(\nu)}
   = \sum_{\kappa\in\widehat{\CK}} N_{\kappa\, \nu-\rho}^{\nu'} \, \hat r_\kappa$.
\end{theorem}
\begin{proof}
(\ref{CI-LR0}) and (\ref{CI-LR0ev}) result from the previous discussion. 
The product $R_n(T) \chi_\nu(T)$ may then be decomposed on characters, 
$$ R_n(T) \chi_\nu(T) =\sum_{\kappa\in\CK} r_\kappa \chi_\kappa(T) \chi_\nu(T)=\sum_{\kappa\in\CK\atop \nu'} N_{\kappa\nu}^{\nu'}\, r_\kappa \chi_{\nu'}(T)=
\sum_{\nu'} c_{\nu'}^{(\nu)}\chi_{\nu'}(T)\,,$$ with $c_{\nu'}^{(\nu)} =  \sum_{\kappa\in\CK} N_{\kappa\nu}^{\nu'} \, r_\kappa$, 
which yields (\ref{CI-LR1}). Similarly, $\widehat{R}_n \chi_{\nu-\rho}=\sum_{\nu'}\hat c_{\nu'}^{(\nu)}\chi_{\nu'}$ with 
$\hat c_{\nu'}^{(\nu)} =  \sum_{\kappa\in\widehat{\CK}} N_{\kappa\,\nu-\rho}^{\nu'} \, \hat r_\kappa$, which gives  (\ref{CI-LR1ev}).
Recall that if either of $\lambda,\mu$ or $\nu$ lies on the boundary of the Weyl chamber, $\alpha$, $\beta$ or $\gamma$ 
has at least two equal components and $\CJ_n(\alpha,\beta; \gamma)=0$.
\end{proof}
Thus, in words, $\CJ_n(\alpha',\beta'; \gamma')$ and $\CJ_n(\alpha,\beta; \gamma)$
may be expressed as  linear combinations of LR coefficients
over ``neighboring" weights $\nu'$ of $\nu$. If Conjecture \ref{conj2} is right, the coefficients $c_{\nu'}^{(\nu)},\, \hat  c_{\nu'}^{(\nu)}$ are also non negative. \\
 {\bf Remark.} Note that even though the function $\CJ_n(\alpha,\beta;\gamma)$ is defined for any 
triple $(\alpha,\beta;\gamma)$, compatible or not, integral or not, equations (\ref{CI-LR1}),(\ref{CI-LR1ev}) hold only for  
triples  $(\alpha',\beta';\gamma')$ or $(\alpha,\beta;\gamma)$ associated with {\it compatible} triples  $(\lambda,\mu;\nu)$.
Recall also
from the previous discussion that for $n$ even, the triple $(\alpha',\beta';\gamma')$ is not {\it integral and compatible} if the triple $(\alpha,\beta;\gamma)$ (or $(\lambda,\mu;\nu)$) is.
\\
 {\bf Comment.} It would be interesting  to invert relations (\ref{CI-LR1},\ref{CI-LR1ev}) and to express
 the LR coefficients $N_{\lambda\mu}^\nu$ as linear combinations of the functions $\CJ_n$ and their derivatives.
 In view of the considerations of \cite{Vergne}, this doesn't seem inconceivable\footnote{Our thanks to 
 Mich\`ele Vergne for pointing to that possibility.}.


{\subsection{Expression of the $R$ and $\widehat{R}$ polynomials}
  \label{Pexpressions}
Here is the essence of the method used to compute $R_n$  and $\widehat{R}_n$, as defined through
(\ref{shiftDelta}), (\ref{shiftDeltaev}).\\
We first introduce two families of functions, defined recursively
$$
f(u,m) = - \frac{1}{m-1} \frac{\partial}{\partial v} f(v,m-1) \vert_{v=u} \quad \text{and}  \quad g(u,m)=  -  \frac{1}{m-1}  \frac{\partial}{\partial v} g(v,m-1) \vert_{v=u}
$$
with (see above  the beginning of sec.\,\ref{polRn}) 
$$
f(u,1)= 
2 u \sum _{m=1}^{\infty } \frac{1}{u^2-(2\pi)^2 m^2}+\frac{1}{u}
=\inv{2\tan(u/2)} 
\quad \text{and} \quad
g(u,1)=
2 u \sum _{m=1}^{\infty } \frac{(-1)^m}{u^2-(2\pi )^2 m^2}+\frac{1}{u}
= \inv{2\sin(u/2)}\,.
$$
$R_n$  and $\widehat{R}_n$ , defined in (\ref{shiftDelta},\ref{shiftDeltaev}), are obtained explicitly by an iterative procedure.
We start from
$${1}/{\widetilde\Delta(u)}= 
\prod_{1\le i < j\le n} \frac{1}{(u_i+\cdots+ u_{j-1})}$$
First we pick a variable in $(\widetilde\Delta(u))^{-1}$, say $u_1$, shift it by $p_1(2\pi)$,
 perform a partial fraction expansion of 
 the rational function $\prod_{2\le j\le n}\inv {u_1+\cdots +u_{j-1}+ p_1 (2\pi)}$ with respect to the variable $u_1$ 
 and make use of the previous identities in the summation over $p_1$. 
 This produces a sum of trigonometric functions of $u_1,\cdots, u_{n-1}$
which are $(2\pi)$ periodic or anti-periodic in each of these variables,  times
rational functions of $u_2,\cdots,u_{n-1}$.
Then iterate with the variable $u_2$, say, shifting it by $p_2 (2\pi)$ etc.
(Of course the order of the variables is immaterial.)
As explained in sec.\,\ref{polRn}, the final result has the general form
$$ \frac{R_n\quad (\mathrm{resp.}\,  \widehat{R}_n) }{ \prod_{1\le i < i'\le n} 2 \sin(\oh(u_i+u_{i+1}+\cdots +u_{i'-1}))} $$
where $R_n$, resp. $\widehat{R}$,  is a (complicated) trigonometric function  of the $u$ variables, 
or alternatively a symmetric trigonometric function of the $t$ variables.
  The latter  is then  recast as a sum of real characters of the matrix $T$. \\
This procedure will be illustrated in sec.\,\ref{Rnlown} on the first cases, for $2\le n\le 6$. }

{\bf Remark}. The reader may have noticed the parallel between this way of computing 
$\widehat{R}_n$
and the computation of $\CJ_n$ in \cite{JB-I}: both rely on an iterative partial fraction expansion, the connection 
between the two being the Poisson formula. As a consequence of this simple correspondence, 
 $\CJ_n(\alpha,\beta;\gamma)$ evaluated for a compatible triple
 and $\widehat{R}_n$ have rational coefficients  with the same least common denominator $\delta_n$,
see below Prop. \ref{quIn}. 
   
\subsection{Consequences of Theorem 1}
\label{conseqTh1}
{\noindent(i) We start with a useful lemma
\begin{lemma}\label{usefullemma}{With the notations of Theorem \ref{CI-LR}, we have the relations
\bea\label{lemma1}\sum_{\kappa\in \CK} r_\kappa \dim V_\kappa &=&1\\
\label{lemma1e}  \sum_{\kappa\in \widehat{\CK}} \hat r_\kappa \dim V_\kappa &=&1\\
\label{lemma2} \sum_{\nu, \nu'}  N_{\lambda\mu}^{\ \nu'} c_{\nu'}^{(\nu)}  \dim V_\nu &=&
\dim V_\lambda \, \dim V_\mu \\
\label{lemma2e}  \sum_{\nu, \nu'}  N_{\lambda-\rho\,\mu-\rho}^{\ \nu'}\, \hat c_{\nu'}^{(\nu)}  \dim V_{\nu-\rho} &=&
\dim V_{\lambda-\rho} \, \dim V_{\mu-\rho} 
\,.\eea 
}\end{lemma}
\begin{proof}
From the relation $R_n(T) = \sum_{\kappa\in \CK} r_\kappa \chi_\kappa(T)$ evaluated at $T=I$, with $R_n(I)=1$, it follows that  $\sum_{\kappa\in\CK} r_\kappa \dim V_\kappa=1$. 
Then 
$$c_{\nu'}^{(\nu)} =\sum_{\kappa\in \CK} N_{\kappa \nu}^{\nu'} r_\kappa= \sum_{\kappa\in \CK} N_{\kappa \nu'}^{\nu} r_\kappa$$
 because of the reality of the irreps of h.w. $\kappa$, hence
\bea\nonumber 
\sum_{\nu, \nu'}  N_{\lambda\mu}^{\ \nu'} c_{\nu'}^{(\nu)}  \dim V_\nu &=&\sum_{\kappa\in \CK} r_\kappa \sum_{\nu'} N_{\lambda\mu}^{\ \nu'}
(\sum_\nu N_{\kappa \nu'}^{\nu} \dim V_\nu) \\ \nonumber
&=&\sum_{\kappa\in \CK} r_\kappa \sum_{\nu'} N_{\lambda\mu}^{\ \nu'} \dim V_{\nu'}\dim V_\kappa
\\ \nonumber
&=& \underbrace{\sum_{\kappa\in \CK} r_\kappa \dim V_\kappa}_{=1} \ \sum_{\nu'} N_{\lambda\mu}^{\ \nu'} \dim V_{\nu'}
=\dim V_\lambda \, \dim V_\mu \,.\eea
The two relations (\ref{lemma1e}) and (\ref{lemma2e}) are proved in the same way.
\end{proof}

\noindent (ii) Localization of the normalization integral of $\CJ_n$.\\
 For two given integral (non negative)  $\alpha$ and $\beta$, 
consider the sum of $\CJ_n(\alpha,\beta; \gamma) \Delta(\gamma)$  over the integral $\gamma$'s inside the 
connected part ${\bH}_{\alpha\beta}$ of the support of $\CJ_n$. 
If either $\alpha$ or $\beta$ is non generic, (\ie has two equal components), 
all $\CJ_n(\alpha,\beta; \gamma) $ vanish. 

Conversely if both $\alpha$ and $\beta$ are
generic,  \ie $\lambda$ and $\mu$ are not on the boundary of the Weyl chamber, 
we make use of 
{   (\ref{dimVla}) and (\ref{CI-LR1ev})     %
\bea\label{above} \!\!\!\!\!\!\!\!\!\!
\sum_{\gamma} \CJ_n(\alpha,\beta; \gamma)   
\frac{\Delta(\gamma)}{\Delta(\alpha)\Delta(\beta)} \Sf(n-1)
&=& 
\sum_{\gamma} \CJ_n(\alpha,\beta; \gamma) 
\frac{\dim V_{\nu-\rho}}{\dim V_{\lambda-\rho}\, \dim V_{\mu-\rho}}   \nonumber \\
{} &=& 
\sum_{\nu, \nu'}  N_{\lambda-\rho\,\mu-\rho}^{\ \nu'} \hat c_{\nu'}^{(\nu)}  \frac{\dim V_{\nu-\rho}}{\dim V_{\lambda-\rho}\, \dim V_{\mu-\rho}} = 1 \eea
by Lemma \ref{usefullemma}. (The $\nu$'s on the boundary of the Weyl chamber, for which $\nu-\rho$ is 
not dominant, do not contribute because of the vanishing of $\CJ_n(\alpha,\beta;\gamma)$.)  
Comparing with (\ref{normIn}), we find that 
\be\label{normInev}
\int_{\bH_{\alpha\beta}}d^{n-1}\gamma\,
 \,\CJ_n(\alpha,\beta;\gamma)  \, \frac{ \Delta(\gamma)}{\Delta(\alpha)\Delta(\beta)} 
=\sum_{\gamma \in \bH_{\alpha\beta} \cap \Z^{n-1} } {\CJ_n(\alpha,\beta; \gamma)} \frac{\Delta(\gamma)}{\Delta(\alpha)\Delta(\beta)} 
= \inv{\Sf(n-1)} \,.
\ee
In others words, the normalization integral of $\CJ_n$ over the sector $\gamma_{n-1}\le \cdots\le \gamma_1$
{\it localizes} over the integral points of that sector.}\\[1pt]

\noindent (iii) Quantization of $\CJ_n$. 
\begin{proposition}\label{quIn}{For any integral compatible triple $(\alpha,\beta;\gamma)$,
$\CJ_n(\alpha,\beta; \gamma)$ is an integral multiple of some rational number  $\delta_n^{-1}$.
}\end{proposition}
\begin{proof} Call $\delta_n$ the  least common denominator of the coefficients $\hat c_{\nu'}^{(\nu)} $ in  (\ref{CI-LR1ev}).
Then we see that $\CJ_n(\alpha,\beta; \gamma)$ is an integral multiple of $1 /\delta_n$.   \end{proof}
Unfortunately we have no general
expression of $\delta_n$ and rely on explicit calculations for low values of $n$:

\begin{center}
$
\begin{array}{r||c|c|c|c|c|c}
n & 2 &  3 &4 &5& 6 & \cdots\\
\hline
\delta_n &1 & 1 & 6  & 360 & 9! &\\
\hline
\end{array}
$
\end{center}

\bigskip

\noindent (iv) Asymptotic behavior. 
The asymptotic regime is read off (\ref{CI-LR0}-\ref{CI-LR1ev}): heuristically, we expect that asymptotically, 
for rescaled weights, the  $t$-integral in the 
computation of $\CJ_n$ will be dominated by $t\approx 0$, 
hence $T\approx I$, for which $R_n=\widehat{R}_n=1$, 
whence  the asymptotic equality, 
for $\lambda,\mu,\nu$ large
\be\label{asympLR}  \CJ_n(\alpha',\beta'; \gamma')  \approx \CJ_n(\alpha,\beta; \gamma) 
\approx N_{\lambda\mu}^{\ \nu}\,.\ee
More precisely,   it is known \cite{Rass}  that, as a function of $\nu'$,  $N_{\lambda\mu}^{\ \nu'}$ 
can be extended to a continuous piecewise polynomial function,
 thus for large $\nu$, one approximates the rhs of (\ref{CI-LR1}) by
 $N_{\lambda\mu}^{\ \nu} \sum_{\nu'} c_{\nu'}^{(\nu)} \approx N_{\lambda\mu}^{\ \nu}$ since the coefficients 
 sum up to 1, again as a consequence of $R_n(I)=1$:
 $$\sum_{\nu'} c_{\nu'}^{(\nu)} =\sum_{\kappa\in \CK} r_\kappa \sum_{\nu'} N_{\kappa \nu}^{\nu'}\buildrel 
 \mathrm{large} \ \nu \over{\approx}\sum_{\kappa\in \CK} r_\kappa  \dim V_\kappa =1$$
 as observed above in (\ref{lemma1}).\\
 We shall see below in sec.\,\ref{LREpols} that (\ref{CI-LR0},\ref{CI-LR1}) enable us to go (a bit) beyond this
  leading asymptotic 
 behavior.

 \bigskip

\noindent (v) Compare Conjecture \ref{conj2} and Lemma \ref{conj1}.\\
We just observe here that Conjecture \ref{conj2} is consistent with 
Lemma  \ref{conj1}. 
 Indeed, if we apply (\ref{CI-LR1}) to an {\it admissible} (hence compatible) triple 
$(\lambda,\mu;\nu)$,  with the assumption that the sum over $\nu'$ includes $\nu$
with a non vanishing coefficient $c_\nu^{(\nu)}$, 
 and using  the non negativity of the other $c_{\nu'}^{(\nu)} $ 
(as stated in Conj. \ref{conj2}), one obtains
$\CJ_n(\alpha',\beta';\gamma') \ge N_{\lambda\mu}^\nu >0$, in agreement with Lemma \ref{conj1}.

\section{On polytopes and polynomials}
\label{LREpols}
The polytopes $\bH_{\alpha\beta}$  and $\CH_{\lambda\mu}^\nu$ considered in this section have been introduced in sec.\,\ref{polytopes}.

\subsection{Ehrhart polynomials}
\label{Epols}
Given some rational polytope ${\CP}$, call $s{\CP}$ the $s$-fold dilation of ${\CP}$, \ie the polytope obtained by scaling by a factor $s$ the vertex coordinates (corners) of ${\CP}$ in a basis of the underlying lattice. 
The number of lattice points contained in the polytope $s{\CP}$ is given by   a quasi-polynomial called the Ehrhart quasi-polynomial
of ${\CP}$, see for example \cite{Stanley}.
It is polynomial for integral polytopes but one can also find examples of rational non-integral polytopes, for which 
it is nevertheless a genuine polynomial. 
We remind the reader that the first two coefficients (of highest degree) of the Ehrhart polynomial 
of a polytope ${\CP}$ of dimension $d$ are given, up to simple normalizing constant factors, by the $d$-volume of ${\CP}$ 
and by the $(d-1)$-volume of the union of its facets; 
the coefficients of smaller degree are usually not simply related to the volumes of the faces of higher co-dimension.
We finally mention the Ehrhart--Macdonald reciprocity theorem: the number of interior points of ${\CP}$, of dimension $d$, is given, up to the sign $(-1)^d$, by the evaluation of the Ehrhart polynomial at the negative value $s=-1$ of the scaling parameter.

\subsection{Littlewood-Richardson polynomials}
\label{LRpols}

It is well known \cite{Heck82, GLS} that multiplicities like the LR coefficients admit a semi-classical 
description for ``large" representations. In the present context,
there is an {\it asymptotic equality} of  the LR multiplicity 
$N_{\lambda\mu}^\nu$, when the weights  $\lambda,\mu,\nu$ are rescaled by a common 
large integer $s$,  with the function $\CJ_n$.  
Here again we assume that the admissible triple $(\lambda,\mu;\nu)$ is generic,  in the sense of 
Definition \ref{defgen}.
Indeed, from (\ref{asympLR}), as $s\to \infty$ 
\be\label{asymp}
N_{s\lambda\, s\mu}^{s\nu}
 \approx \CJ_n(\ell(s\lambda+\rho), \ell(s\mu+\rho); \ell(s\nu+\rho))
 \approx  \ \CJ_n(s\alpha,s\beta; s\gamma) =
s^{(n-1)(n-2)/2} \,\CJ_n(\alpha,\beta; \gamma)
\,.\ee
The last equality just expresses the homogeneity of the function $\CJ_n$.

These scaled or ``stretched" LR coefficients have been proved to be polynomial (``Littlewood-Richardson polynomials'') in the stretching parameter  $s$ \cite{DW,Rass}, 
\be\label{stretch} N_{s\lambda\, s\mu}^{s\nu} =P_{\lambda\mu}^\nu (s)\ee  and it has been conjectured that the polynomial $P_{\lambda\mu}^\nu (s)$ (of degree at most $(n-1)(n-2)/2$ by 
(\ref{asymp})), has non negative rational coefficients  \cite{KTT04}. More properties of $P_{\lambda\mu}^\nu (s)$,
namely their possible factorization  and bounds on their degree have been discussed in \cite{KTT}.
For a generic triple, our study leads to an explicit value (eq.~(\ref{asymp})) for the coefficient of highest degree,  namely the kernel function $\CJ_n(\alpha,\beta;\gamma)$, see eq.~(\ref{In}).

{}From the very definition of the hive polytope  $\CH_{\lambda\,\mu}^\nu$  associated with an admissible triple (each integral point of  
which is a honeycomb contributing to the multiplicity), with  Littlewood-Richardson, or stretching, polynomial $P_{\lambda\mu}^\nu (s)$,  and from the general definition of the Ehrhart polynomial, it is clear that both polynomials are equal.
Notice that $P_{\lambda\mu}^\nu (s)$, defined as the Littlewood-Richardson polynomial of the triple $(\lambda,\mu;\nu)$ or as the Ehrhart polynomial of the polytope $\CH_{\lambda\,\mu}^\nu$, 
is  polynomial even if the hive polytope happens not to be an integral polytope;  on the other hand the Ehrhart polynomial of the polytope defined as the convex hull of the integral points 
of $\CH_{\lambda\,\mu}^\nu$ will differ from $P_{\lambda\mu}^\nu (s)$ if $\CH_{\lambda\,\mu}^\nu$ is not integral, see
two examples in sec.\,\ref{exampleSU5} and \ref{exampleSU6}.

{}From the volume interpretation of the first Ehrhart coefficient, which was recalled in sec.\,\ref{Epols}, we find:
 \begin{proposition} \label{Inasavol}{For $\SU(n)$, the normalized $d$-volume
  $\CV$ of the hive polytope $\CH_{\lambda\mu}^\nu$ equals 
 $ d! \;\CJ_n(\alpha,\beta;\gamma)$, with $d=(n-1)(n-2)/2$,  for a generic and admissible triple $(\lambda,\mu;\nu)$,  with $\alpha = \ell(\lambda)$, $\beta = \ell(\mu)$, $\gamma = \ell(\nu)$, and with $\CJ_n(\alpha,\beta; \gamma)$ given by eq.~(\ref{In}).}
 \end{proposition}
{We use here the definition given by \cite{HaaseEtAl, Magma}: for a polytope of dimension $d$, the Euclidean volume $v$  is related to the normalized volume  $\CV$ by $v=\CV/d!$. More generally the total normalized $p$-volume $\CV_p$ of the $p$-dimensional faces of a polytope is related to its total Euclidean $p$-volume $v_p$ by $v_p=\CV_p/p!$. }

This is  consistent with the result\ \cite{KT99}  that the LR coefficient is equal to the number of integral points in the  hive polytope. 
In words, (\ref{asymp}) says that the number  of integral points  of that polytope  is asymptotically well approximated by its euclidean volume $\CJ_n$.\\

{The Blichfeldt inequality \cite{Blichfeldt}  valid for an integral polytope  ${\mathcal Q}$ of dimension $d$, states that its number of integral points is smaller than  $\CV + d$, where $\CV$ is its normalized volume.
This property, which a fortiori holds for a rational polytope  ${\mathcal H}$ with integral part ${\mathcal Q}$, 
together with Proposition 4, 
implies the following inequality for a generic hive polytope $\CH_{\lambda\,\mu}^\nu$ of $\SU(n)$: 
\be\label{blichfeldt} 
d! \;\CJ_n(\alpha,\beta;\gamma)  \ge N_{\lambda\mu}^\nu-d
\ee
 with $d=(n-1)(n-2)/2$ and  $\alpha=\ell(\lambda)$, $\beta=\ell(\mu)$, $\gamma=\ell(\nu)$.}

\subsection{Polytopes versus symplectic quotients}
\label{SympQuo}
 
Here is another argument relating the volume of the hive polytope with $p(\gamma|\alpha, \beta)$, hence also with $\CJ_n(\alpha,\beta;\gamma)$, 
for $\alpha = \ell(\lambda)$, $\beta = \ell(\mu)$, $\gamma = \ell(\nu)$, $\lambda, \mu, \nu$
being dominant integral weights.
It goes in two steps, as follows.\\
{\it Step 1.}\\
 $N_{\lambda \mu}^{\nu}$ is the number of integral points of the hive polytope.\\
For large $s$, the coefficient $N_{s\lambda\, s\mu}^{s\nu}$ is approximated by $s^d$ times the volume of the same polytope.\\
{\it Step 2.}\\
For large $s$, $N_{s\lambda\, s\mu}^{s\nu}$ is approximated\footnote{More precisely $\lim_{s \rightarrow \infty} \frac{1}{s^d} N_{s \lambda\, s\mu}^{s \nu} 
= \int \omega^d/d!$, with $d=(n-1)(n-2)/2$, where $\omega$ is the symplectic 2-form on the symplectic and K\"ahler manifold of complex dimension $d$ defined as
$({\mathcal O}_\lambda \times {\mathcal O}_\mu \times {\mathcal O}_{\overline{\nu}}) // \SU(n) {\, := \, } m^{-1}(0)/ \SU(n)$, with $m$, the moment map $m: (a_1, a_2, a_3) \in {\mathcal O}_\lambda \times {\mathcal O}_\mu \times {\mathcal O}_{\overline{\nu}} \mapsto a_1+a_2+a_3 \in Lie(\SU(n))^*.$ } 
by the volume of a symplectic quotient of the product of three coadjoint orbits labelled by $\lambda, \mu, \overline{\nu}$, where $\overline{\nu}$ is the conjugate of $\nu$.
\\The same volume is given, up to known constants,  by $p(\gamma \vert \alpha, \beta)$, hence by $\CJ_n(\alpha,\beta;\gamma)$, see \cite{KT00}, Th4.\\
Hence the result.

As already commented in \cite{KT00},  the equality between the two volumes is quite indirect and it would be nice to construct a measure preserving map between the hive polytope 
and the above symplectic quotient, or a variant thereof. To our  knowledge, this is still an open problem.

The details of the first part of step 2 are worked out in \cite{SuTa}. 
We should mention that this last reference also adresses the problem of calculating  the function $p(\gamma|\alpha, \beta)$, at least when the arguments are determined by dominant integral weights, and the authors present quite general formulae that are similar to ours. However, they do not use the  explicit writing of the orbital measures using
formula (\ref{HCIZ}),
which was a crucial ingredient of our approach and allowed us to obtain rather simple expressions for $\CJ_n(\alpha,\beta;\gamma)$.

\subsection{Subleading term}
\label{subleadingterms}
\medskip
From the asymptotic behavior (\ref{asymp}), we have 
\be  N_{s\lambda\,s \mu}^{s \nu} =  \nonumber P_{\lambda\mu}^\nu(s)=
s^{(n-1)(n-2)/2} \CJ_n(\ell(\lambda),\ell(\mu);\ell(\nu) ) (1+O(s^{-1}))\ee
provided the leading coefficient $\CJ_n(\ell(\lambda),\ell(\mu); \ell(\nu))$ does not vanish.
According to Lemma \ref{conj1}
 the stretching polynomial $P_{\lambda\mu}^\nu(s)$ is  
of degree $(n-1)(n-2)/2$ for $\nu$ inside the tensor polytope and for $\lambda,\mu \notin \partial C$, but is 
of lower degree on the boundary of that polytope, or for $\lambda$ or $\mu$ on $\partial C$.

 Write (\ref{CI-LR1}) for stretched weights
\[  \CJ_n(\ell(s \lambda+\rho),\ell(s \mu+\rho); \ell(s\nu+\rho) ) = 
\sum_{\kappa\in \CK} r_\kappa \sum_{\nu'}N_{s \nu\, \kappa}^{\nu'} N_{s\lambda\, s\mu}^{\nu'}\, .\]
 For $s$ large enough, all the weights $\nu'= s\nu+ k$, where $k$ runs over 
the multiset $\{\kappa\}$ of weights (\ie counted with their multiplicity)
of the irrep with highest weight $\kappa$, 
are dominant and thus contribute to the 
multiplicity $N_{s \nu\, \kappa}^{\nu'}$~\cite{Racah-Speiser}.
Thus
\be \label{CJLRas}\CJ_n( \ell(s \lambda+\rho),\ell(s \mu+\rho); \ell(s\nu+\rho) ) = 
\sum_{\kappa\in \CK} r_\kappa \sum_{k \in \{\kappa\} }N_{s\lambda\, s\mu}^{s\nu +k}\,. \ee
But as a function of $\lambda,\mu,\nu$, and in the case of $\SU(n)$, the LR coefficient $N_{\lambda\,\mu}^\nu$ is itself
a piecewise polynomial \cite{Rass}: more precisely in the latter reference it is shown that, for the case of $\SU(n)$, the quasi-polynomials giving the Littlewood-Richardson coefficients in the cones of the Kostant complex are indeed polynomials of total degree at most $(n-1)(n-2)/2$ in the three sets of variables defined as the components of the highest weights $\lambda, \mu, \nu$.\\
{\bf Remark.} {The well known Kostant--Steinberg 
method for the evaluation of the LR coefficients (a method where one performs a Weyl group average over the Kostant function) 
is not used in our paper,  or it is only used as a check. However we should stress that, even in the case of SU(3) where the LR coefficients can be deduced from our kernel function $\CJ_3$, see below sec. \ref{caseSU3}, 
the expressions obtained for $N_{\lambda\mu}^\nu$ using the  Kostant--Steinberg 
method differ from ours.}\\
 If we assume  that $N_{\lambda\,\mu}^\nu$ 
may be extended to a function of the same class as $\CJ_n$, namely $C^{\,n-3}$, see above
sec.\,\ref{discIn}, a Taylor expansion to second order of the rhs of (\ref{CJLRas}) 
is possible for $n\ge 4$. This leaves  out the cases $n=2$ and $n=3$ which may be 
treated independently, see below sec. \ref{caseSU2} and \ref{caseSU3}.
We thus Taylor expand for large $s$
  \bea\nonumber
 \CJ_n( \ell(s \lambda+\rho),\ell(s \mu+\rho);\ell(s\nu+\rho) )& = &
\sum_{\kappa\in \CK} r_\kappa \sum_{k \in \{\kappa\} }P_{\lambda\, \mu}^{\nu +k/s}(s)\\
\nonumber 
&=& 
 \sum_{\kappa\in \CK}  r_\kappa \(  \dim V_\kappa\, P_{\lambda\, \mu}^\nu(s)  + \inv{s}  \sum_{k \in \{\kappa\}} k \nabla_\nu P_{\lambda\, \mu}^\nu(s)
 +\cdots  \)\\ \label{asymp2}
 &=&  P_{\lambda\, \mu}^\nu(s) \(1+  o\Big(\inv{s}\Big) \)
  \eea
since $ \sum_{\kappa\in \CK}  r_\kappa  \dim V_\kappa=1$ as noticed above in sec.\,2.2, and $ \sum_{k \in \{\kappa\}} k =0$ 
in any irrep. 
Thus for generic points, the two  polynomials $\CJ_n( \ell(s \lambda+\rho),\ell(s \mu+\rho); \ell(s\nu+\rho) )$  and 
$P_{\lambda\mu}^\nu(s)$ have the 
same two terms of  highest degree $d_{max}=(n-1)(n-2)/2$ and $d_{max}-1$.
 In the degenerate case where the term of degree $d_{max}$ vanishes and the next does  not,
the leading terms of degree $d_{max}-1$ are equal. 
If the degree is strictly lower than $d_{max}-1$,  there is no obvious relation between the
two polynomials, see examples at the end of sec.\,\ref{stretch2}. 
 
\section{A case by case study for low values of $n$}

We examine in turn the cases $n=2,\cdots, 6$. \\
 
\subsection{Expression and properties of the ${\CJ}_n$ function}
  
The expressions of  ${\CJ}_2,\ {\CJ}_3$ and $\CJ_4$ were already given in \cite{JB-I}. 
  We repeat them below for the reader's convenience. 
  Those of $\CJ_5$ and $\CJ_6$, which   are fairly cumbersome,  are available on 
  the web site  \url{http://www.lpthe.jussieu.fr/~zuber/Z_Unpub.html} 
  
\subsubsection{The case of SU(2)}
 \label{caseSU2}
 
In the case of $n=2$, the function $\CJ_2$ reads
\be\label{CI2} 
\CJ_2( \alpha, \beta; \gamma)= (\bun_I(\gamma_{12}) -\bun_{-I}(\gamma_{12})) \ee
where $\gamma_{12}:=\gamma_1-\gamma_2$ and $\bun_I$ is the characteristic function of the segment
\footnote{This result should be connected with the fact that the  support of the convolution product of  measures on concentric 2-spheres is an annulus. }
$I=(|\alpha_{12}-\beta_{12}|, \alpha_{12}+\beta_{12})$.
Then, when evaluated for shifted weights, $\alpha'=\alpha_{12}+1=\lambda_1+1$, $\beta'=\beta_{12}+1=\mu_1+1$, 
$\gamma'=\gamma_{12}+1=\nu_1+1>0$, 
it takes the value 1 iff 
$|\alpha_{12}-\beta_{12}|< \gamma_{12}+1<  \alpha_{12}+\beta_{12} +2$, \ie
iff $|\alpha_{12}-\beta_{12}|\le \gamma_{12}\le  \alpha_{12}+\beta_{12} $ which is precisely 
 the well known value of the LR coefficient, 
$$ N_{\lambda\mu}^\nu= \begin{cases} 1 & \rm{if}\ |\alpha_{12}-\beta_{12}|=|\lambda_1-\mu_1| \le \gamma_{12}=\nu_1\le 
  \alpha_{12}+\beta_{12} =\lambda_1+\mu_1\ \\ & \  \ \rm{and}\ \nu_1-|\lambda_1-\mu_1| \ \rm{even} \\
0 & \rm{otherwise}\end{cases}\,.$$
We conclude that 
\be\label{I2shift} \CJ_2(\alpha',\beta'; \gamma')= N_{\lambda\mu}^{\nu}\,,\ee
in agreement with the general formula  (\ref{CI-LR1}), provided we assume that the
indicator function  vanishes at the end points of the interval $I$.\\
On the other hand,  as we shall see below in sec. \ref{R2-R3}, $\widehat{R}_2=\oh \chi_1(T)$,
so that   (\ref{CI-LR1ev}) amounts to 
\bea \CJ_2(\alpha,\beta; \gamma)&=& \oh \sum_{\nu'} N_{\lambda_1-1\, \mu_1-1}^{\nu'} N_{\nu_1-1\, 1}^{\nu'}\\
&=&   \begin{cases} 1  & \rm{if}\ |\lambda_1-\mu_1|+2 \le \gamma_{12}=\nu_1\le 
\lambda_1+\mu_1-2\ \\ 
  & \  \ \rm{and}\ \nu_1-|\lambda_1-\mu_1| \ \rm{even} 
 \\
 \oh   & \rm{if}\  \nu_1= |\lambda_1-\mu_1| \ \mathrm{or}\ = \lambda_1+\mu_1 \\
 0 & \rm{otherwise}
\end{cases}
\eea
which is consistent with (\ref{CI2}) if we assume  now that the 
indicator function  takes the value $\oh$ at the end points of the interval $I$. This rather peculiar 
situation is a consequence of the irregular, discontinuous, structure of $\CJ_2$.

\subsubsection{The case of SU(3)}
 \label{caseSU3}
For $n=3$, $\CJ_3$ takes a simple form within the  tensor polytope (here a polygon). In \cite{JB-I}, the following was established.

The function
\be\label{oldI3}  \CJ_3(\alpha,\beta; \gamma) =\inv{4}\sum_{P,P'\in S_3} \varepsilon_{PP'} \,  \epsilon(A_1)(|A_2|-|A_1-A_2|)\,,\ee
with $A_1$ and $A_2$ as in (\ref{Aj}), 
may be recast in a more compact form:
\begin{proposition} 
\label{CI3}Take $\alpha_1\ge\alpha_2\ge \alpha_3$, and likewise for $\beta$. For 
$\gamma$ satisfying  (\ref{equtr}), Horn's inequalities and $\gamma_1\ge \gamma_2\ge \gamma_3$, 
\bea\label{PDF3b} \CJ_3(\alpha,\beta; \gamma) 
 &=&
\inv{6}(\alpha_1-\alpha_3+\beta_1-\beta_3+\gamma_1-\gamma_3) -\oh|\alpha_2+\beta_2-\gamma_2|  -\inv{3}\psi_{\alpha\beta}(\gamma)-\inv{3}\psi_{\beta\alpha}(\gamma)\eea
where 
\be\label{psi} \psi_{\alpha\beta}(\gamma)=
\begin{cases}  (\gamma_2-\alpha_3-\beta_1)-(\gamma_1-\alpha_1-\beta_2)   \quad \mathrm{if}\ \gamma_2-\alpha_3-\beta_1 \ge 0 \ \mathrm{and}\ \gamma_1-\alpha_1-\beta_2<0 \\
 (\gamma_3-\alpha_2-\beta_3)-(\gamma_2-\alpha_3-\beta_1)   \quad \mathrm{if}\ \gamma_3-\alpha_2-\beta_3 \ge 0 \ \mathrm{and}\ \gamma_2-\alpha_3-\beta_1<0 \\
  (\gamma_1-\alpha_1-\beta_2)-(\gamma_3-\alpha_2-\beta_3)   \quad \mathrm{if}\ \gamma_1-\alpha_1-\beta_2 \ge 0 \ \mathrm{and}\ \gamma_3-\alpha_2-\beta_3<0
    \end{cases}\,.
    \ee
    $\CJ_3(\alpha,\beta; \gamma)$ takes non negative values inside the  tensor polygon and vanishes
    by continuity along the edges  of the polygon. It also vanishes whenever two components of $\alpha$
    or $\beta$ coincide (non generic orbits). 
    \end{proposition}
    
The non-negativity follows from the interpretation of $\CJ_3$ as proportional with a positive coefficient
    to the PDF $p$.

\bigskip
 Consider now an admissible triple  $(\lambda,\mu;\nu)$ of  highest weights of $\SU(3)$.
 The associated triple $(\alpha,\beta;\gamma)$ is defined as explained above, 
 $\alpha_1=\lambda_1+\lambda_2,\ \alpha_2=\lambda_2,\,\beta_1=\mu_1+\mu_2,\,\beta_2=\mu_2,\ \alpha_3=\beta_3=0$, $\gamma_1=\nu_1+\nu_2+\nu_3, \gamma_2=\nu_2+\nu_3$ 
and $\gamma_3=\nu_3=\inv{3}(\lambda_1+2\lambda_2+ \mu_1+2\mu_2-\nu_1-2\nu_2)$, an integer, 
so that $\sum_{i=1}^3 (\gamma_i-\alpha_i-\beta_i) =0$. Then  
\begin{proposition} 
\begin{enumerate} 
\item  For an admissible triple, the function $\CJ_3(\alpha,\beta;\gamma)$ of eq.\,(\ref{PDF3b})  takes only values that are  integral  
and  non negative; as just discussed, 
these values vanish by continuity along the edges of the polygon;  
the vertices of the boundary polygon  are integral  and give admissible $\gamma$'s;
\item for $\alpha=\ell(\lambda),\ \beta=\ell(\mu),\ \gamma=\ell(\nu)$, 
$\CJ_3(\alpha,\beta;\gamma)= N_{\lambda\mu}^\nu-1$; in particular, if some $\lambda_i$ or
$\mu_i$ vanishes, hence $\alpha$ or $\beta$ are non generic, $N_{\lambda\mu}^\nu=1$, a well-known
property of SU(3); 
\item the points $\nu$ of value $\CJ_3(\ell(\lambda),\ell(\mu);\ell(\nu))=m$, for $0\le m<m_{max}$ form a ``matriochka" pattern,   
{see Fig. \ref{polygon9565}}.
\item Now evaluate $\CJ_3$  at shifted weights $\lambda'=\lambda+\rho$, $\mu'=\mu+\rho$, $\rho$ the Weyl vector $(1,1)$, hence
$\alpha'_i= \ell_i(\lambda)+3-i$, $\beta'_i=\ell_i(\mu)+3-i$ and still $\alpha'_3=\beta'_3=0$. Then
\be\label{I3shift} \CJ_3(\alpha',\beta'; \gamma')=N_{\lambda\mu}^{\nu}\ee
 with $\nu$ such that $\gamma'_i=\ell_i(\nu)+3-i$,  $i=1,2,3$. 
 \item The sum $\sum_{ \gamma\in \bH_{\alpha\beta} \cap \Z^2}  \frac{\Delta(\gamma)}{ \Delta(\alpha)\Delta(\beta)} \CJ_3(\alpha,\beta;\gamma)$ equals $\oh$; therefore replacing the sum by an integral over the domain $\gamma_3\le \gamma_2\le \gamma_1$, see (\ref{normIn}), gives the same value (namely $\oh$).
\end{enumerate} 
\end{proposition}

\begin{figure}[!tbp]
  \centering
     \includegraphics[width=16pc]{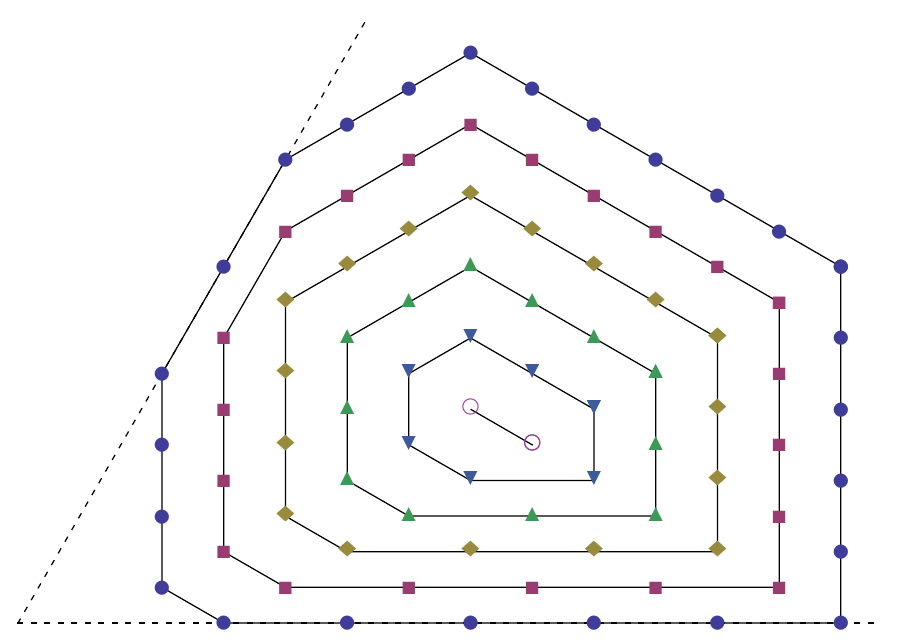}
    \caption{\label{polygon9565}  The Horn-tensor polygon $\bH_{\alpha\beta}={\mathbf H}_{\lambda\mu}$ for 
    the two SU(3) weights $\lambda=(9,5)$  $\mu=(6,5)$, hence $\alpha=(14,5,0),\ \beta=(11,5,0)$. 
     The multiplicity increases from 1 to 6 inside the polygon, giving a matriochka pattern to the successive contours.}
\end{figure}

\begin{proof}
Point 1 follows from  Proposition \ref{quIn}, with $\delta_3=1$.
Integrality of the vertices of the polygon is seen by inspection of Horn's inequalities.
Point 4 follows from  (\ref{CI-LR0}) together with the fact that for $n=3$, the polynomial $R_3=1$, see
below sec.\,\ref{Rnlown}. 
Points 2  follows from (\ref{I3shift}) and  
the observation made in \cite{CZ14} that, {\bf for SU(3)},
\be\label{SU3proper}N_{\lambda+\rho\,\mu+\rho}^{\nu+\rho}= N_{\lambda\,\mu}^\nu +1\,. \ee
The matriochka pattern of point 3 matches the similar 
pattern of points of multiplicity $m+1$ in the tensor product decomposition $\lambda\otimes \mu$ (cf \cite{CZ14},
 eq (22)]).
 Point 5 has already been derived in sec.\,\ref{conseqTh1} and
is here a direct  consequence of 
$\label{dimdim}\sum_\nu N_{\lambda\,\mu}^\nu \dim V_\nu=\dim V_\lambda \dim V_\mu\,.  $
\end{proof}

We want to stress a remarkable consequence of the above eq. (\ref{PDF3b},\ref{psi},\ref{I3shift})
\begin{corollary}
\label{LRSU3} The LR coefficients $N_{\lambda\mu}^\nu$ 
of SU(3) may be expressed as a piecewise linear function 
of the weights $\lambda,\mu,\nu$, sum of the four terms of (\ref{PDF3b}). 
\end{corollary}

To the best of our knowledge, this expression  was never given before.
Note that the lines of non differentiability of the expression 
(\ref{PDF3b}) split the plane into at most 9 domains. In each domain, the function $\CJ_3$ is linear. 
This is to be contrasted with the  known expressions that follow from 
Kostant--Steinberg formula (see for example \cite{FultonHarris}, Prop. 25-29) and which involve a sum over two copies of the $\SU(3)$ Weyl group. \\
We should also recall that there exist yet another formula for the multiplicity $N_{\lambda\mu}^\nu$, stemming from
its interpretation \cite{KT99} as the number of integral solutions to the inequalities on the honeycomb variable,
\bea\label{PDF3c} N_{\lambda\,\mu}^\nu&=& \CJ_3(\alpha',\beta';\gamma')= \min(\alpha'_1,-\beta'_3+\gamma'_2,\alpha'_1+\alpha'_2+\beta'_1-\gamma'_1)\\
&& \nonumber -
\max(\alpha'_2,\gamma'_3-\beta'_3,\gamma'_2-\beta'_2, \alpha'_1+\alpha'_3+\beta'_1-\gamma'_1,
\alpha'_1+\alpha'_2+\beta_2'-\gamma'_1,\alpha'_1-\gamma'_1+\gamma'_2)\\
&=& \nonumber 1+ \min(\lambda_1+\lambda_2, \nu_2+\sigma, \nu_2-\mu_2+2\sigma)\\
&& \nonumber -\max(\lambda_2,\sigma,\nu_2-\mu_2+\sigma, \nu_2-\lambda_2-\mu_2+2\sigma,  \nu_2-\mu_1-\mu_2+2\sigma,\lambda_1+\lambda_2-\nu_1)  \
 \,, \eea
where $\sigma:=\inv{3}(\lambda_1+2\lambda_2+\mu_1+2\mu_2-\nu_1-2\nu_2)$. 
See also \cite{BMW, CZ14} for alternative and more symmetric formulae and \cite{CZ16} for an expression 
in terms of a semi-magic square.

\noindent {\bf Remark}. The  lines or half-lines of non-differentiability of $\CJ_3$, 
as they appear on expression (\ref{PDF3b}), (see also Figures in \cite{JB-I}), are a  {\it subset} of the lines along which 
two arguments of the $\min$ or of the $\max$ functions of (\ref{PDF3c}) coincide.

\subsubsection{The case of SU(4)}
\label{case4}
 The case of SU(4) is more complicated.  
Some known features of SU(3) are no longer true. In particular, 
it is generically not true that multiplicities $N_{\lambda\mu}^\nu$ are equal to 1 on the boundary of the polytope; 
 there is no matriochka pattern, with multiplicities growing as one goes deeper inside the tensor polytope; 
and relation (\ref{SU3proper}) is wrong and meaningless, since $(\lambda+\rho,\mu+\rho;\nu+\rho)$ cannot be 
compatible  if $(\lambda,\mu;\nu)$ is. 

We first recall the expression of $\CJ_4(\alpha,\beta;\gamma)$ given in \cite{JB-I}. With 
$A_j$ standing for $A_j(P,P',P'')$ in the notations of (\ref{Aj}), 
\begin{eqnarray}
\label{CI4}
\!\!\!\!\!\! {\CJ}_4(\alpha,\beta;\gamma ) &=& \inv{2^3 4!} \!\! \sum_ {P, P^\prime, P^{\prime \prime} \in  S_4} \varepsilon_P \varepsilon_{P^\prime} \varepsilon_{ P^{\prime \prime}}\,
 \epsilon(A_1) \(\frac{1}{3!} \epsilon(A_2-A_1)(\vert A_3 - A_1\vert^3 - \vert A_3-A_2+A_1\vert^3 - \vert A_3-A_2 \vert^3 + \vert A_3\vert^3)\right. \cr
{} &{}& \left.-\frac{1}{3}  \epsilon(A_2)(\vert A_3 \vert^3 - \vert A_3 - A_2\vert^3) - \frac{1}{2} (\vert A_2 - A_1\vert - \vert A_2\vert)
(\vert A_3 - A_2\vert (A_3-A_2) + \vert A_3\vert A_3)\)\,.
\end{eqnarray}
One can actually restrict the previous triple sum over the Weyl group to a double sum only  while multiplying the obtained result by $4!$, and this is quite useful for practical calculations.

 Then, we have,   
for an admissible triple $(\lambda,\mu;\nu)$ of h.w. of $\SU(4)$ (with $\lambda,\mu\notin \partial C$,
\ie $\lambda_i, \mu_i\ne 0$), 
 and $\alpha=\ell(\lambda),\, \beta=\ell(\mu),\, \gamma=\ell(\nu)$, \\[-12pt]
\begin{proposition}
\begin{enumerate}
\item   $N_{\lambda\mu}^\nu \ge 4 $ inside the tensor polytope.
\item 
 $\CJ_4(\alpha,\beta;\gamma)$ vanishes when $\gamma$ belongs to the faces of the polytope $\bH_{\alpha\beta}$; 
  conversely  $\CJ_4$ does not  vanish {\it inside} the polytope.
\item At these interior points,   $6 \CJ_4(\alpha,\beta;\gamma)$, which is the normalized 3-volume $\CV$
of the hive polytope $\CH_{\lambda\mu}^\nu$, is an integer. 
\item  That integer satisfies
$\CV=6\CJ_4(\alpha,\beta;\gamma)\ge N_{\lambda\mu}^\nu -3$. \item 
The sum $\sum_{\gamma\in \bH_{\alpha\beta} \cap \Z^3 }  \frac{\Delta(\gamma)}{ \Delta(\alpha)\Delta(\beta)} \CJ_4(\alpha,\beta;\gamma)$  
equals $\inv{12}$, which matches the normalization (\ref{normIn}). 
\end{enumerate}
\end{proposition}
\begin{proof}
Point 1 results from a general inequality in integral $d$-polytopes that asserts that their number of integral
points is larger or equal to $d+1$, see \cite{BDLDPS}, Theorem 3.5. Here for points $\nu$ inside the tensor polytope,
the polytope $\CH_{\lambda\,\mu}^\nu$ is integral and 3-dimensional, hence $d=3$.
The first part of point 2 has been 
already amply discussed,  while the second one follows from  Lemma \ref{conj1}. 
Points 3 and 5 have been established in sec.\,\ref{conseqTh1}. 
Point 4 follows from Blichfeldt's inequality (\ref{blichfeldt}).
\end{proof}

The consequences of Theorem 1 on the values of $\CJ_4$ at shifted weights will be discussed in the next subsection.

\subsubsection{A few facts about SU(5)}
 \begin{enumerate}
\item 
Based on the study of numerous examples, it seems that for weights $\nu$ interior to the tensor polytope, we have the lower bound
$N_{\lambda\mu}^\nu \ge 8$. Note that the afore mentioned inequality of Theorem 3.5 of \cite{BDLDPS} 
(which would give the weaker $N_{\lambda\mu}^\nu \ge 7$) is no longer 
applicable, since the hive polytope is not generally integral for $n=5$, see a counter-example in sec. \ref{exampleSU5}. 
\item 
$\CJ_5(\alpha,\beta;\gamma)$ vanishes outside (and on the boundary) of the polytope, as already discussed. 
\item 
For a compatible triple $(\alpha,\beta;\gamma)$ and $\gamma$ inside the polytope $\bH_{\alpha\beta}$,
  $360 \CJ_5(\alpha,\beta;\gamma)  $ 
  is a positive 
integer  (see sec.\,\ref{conseqTh1}),  provided $\alpha$ and $\beta$ have only distinct components.  
It is non vanishing according to Lemma \ref{conj1}.
 Moreover  $N_{\lambda\mu}^\nu\le 6!\, \CJ_5(\alpha,\beta;\gamma)   +6$   according to (\ref{blichfeldt}).
\item 
$\sum_{\gamma\in \bH_{\alpha\beta} \cap \Z^4 } 
\CJ_5(\alpha,\beta;\gamma) \frac{\Delta(\gamma)}{\Delta(\alpha)\Delta(\beta)}= \inv{288}$, 
 see (\ref{normIn}) again.
\end{enumerate}
  
\subsection{The polynomials $R_n$ and $\widehat{R}_n$. Application of Theorem 1}
\label{Rnlown}
 As in section \ref{WHCIZ} the notation $\chi_\lambda$ denotes the character of the Lie group $\SU(n)$
 associated with the irrep of highest weight $\lambda$. Also recall that for $n$ odd, $\widehat{R}_n=R_n$.

  \subsubsection{Cases $n=2$ and $n=3$}
  \label{R2-R3}
  For $n=2$ and $n=3$, the polynomial $R_n$  is equal to 1. Indeed:
  \bea \sum_{p_1=-\infty}^\infty \frac{(-1)^{p_1}}{u_1+2\pi p_1}&=&\inv{2\sin(u_1/2)} =\frac{\ii}{\Delta(e^{\ii t_j})}\\
\nonumber \sum_{p_1,p_2=-\infty}^\infty \frac{1}{(u_1+2\pi p_1)(u_2+2\pi p_2)(u_1+u_2+2\pi (p_1+p_2))}&=&\inv{2^3\sin(u_1/2)\sin(u_2/2)\sin((u_1+u_2)/2)}\\ &=& \frac{\ii^3}{\Delta(e^{\ii t_j})}\,. \eea
On the other hand, 
$$ P.V. \sum_{p_1=-\infty}^\infty \frac{1}{u_1+2\pi p_1}=\inv{u_1}+\sum_{p_1=1}^\infty \frac{2u}{u_1^2-(2\pi p_1)^2}
=\frac{\cos(u_1/2)}{2\sin(u_1/2)} =\frac{\oh \ii \tr T}{\Delta(e^{\ii t_j})}\,,$$
hence $\widehat{R}_2(T)=\oh \chi_{1}(T)$, while  $\widehat{R}_3=R_3 = 1$.

\subsubsection{Case $n=4$}
\label{R4}
In contrast, for $n\ge 4$, one finds non trivial polynomials $R_n(T)$ and ${\widehat R}_n(T)$.
For instance for $n=4$,  with the notations   $D_4$,  ${\widehat D}_4$  and $\varpi_4$ introduced in (\ref{defpin})
\bea \nonumber\!\!\!\!\!\!\!\!\!\!\!\!\!\!\!\!
{ D_4}&=&\sum_{p_1, p_2, p_3 =-\infty}^\infty \prod_{1\le i < i'\le 4}   \frac{(-1)^{p_1+p_3}}{u_i+u_{i+1}+\cdots +u_{i'-1}+(p_i+\cdots +p_{i'-1})(2\pi)} \\ \nonumber
&=& \frac{ \inv{12}\Big( 6 +\sum_{1\le i < j\le 4}\cos\oh(u_i+\cdots+u_{j-1}) \Big)}{\varpi_4}
= \frac{\inv{24}(\tr T \tr T^\star +8)}{  \varpi_4}\\
& =&{ \ii^6 } \frac{\inv{24}(\tr T \tr T^\star +8)}{\Delta(e^{\ii t_i})}\,,  \eea
and likewise
\bea \nonumber\!\!\!\!\!\!\!\!\!\!\!\!\!\!\!\!
{ {\widehat D}_4}&=&\sum_{p_1, p_2, p_3 =-\infty}^\infty \prod_{1\le i < i'\le 4}   \frac{1}{u_i+u_{i+1}+\cdots +u_{i'-1}+(p_i+\cdots +p_{i'-1})(2\pi)} \\ \nonumber
&=& \frac{ \inv{3}
 \left(2 \cos \left(\frac{{u_2}}{2}\right) \cos
   \left(\frac{{u_1}}{2}+\frac{{u_2}}{2}+\frac{{u_3}}{2}\right)+\cos
   \left(\frac{{u_1}}{2}-\frac{{u_3}}{2}\right)\right)}{\varpi_4} 
   =\frac{\inv{6}  
   \sum_{1\le i<j\le 4} e^{\ii (x_i+x_j)}}{\varpi_4} 
      \\
   &=& { \ii^6 } \frac{\inv{12}((\tr T )^2 - \tr T^2)}{\Delta(e^{\ii t_i})}
   \eea
 hence
\bea\label{P4}R_4(T)&=&\inv{24}(\tr T \,\tr T^\star +8 )= \inv{24}(9 +  \chi_{(1,0,1)}(T))\,,\\
\widehat{R}_4(T)&=&{\inv{12}((\tr T )^2 - \tr T^2)} =\inv{6} \chi_{(0,1,0)}(T)\,. \eea
\def\balpha{\hat{\alpha}}
\def\bbeta{\hat{\beta}}
Now, in SU(4), we can write
\bea\nonumber \chi_{(1,0,1)}(T)\chi_\nu(T)&=&\chi_\nu(T) + \sum_{\nu' }  \chi_{\nu'}(T)\\
\nonumber  \chi_{(0,1,0)}(T)\chi_{\nu-\rho}(T)&=& \sum_{\nu'' }  \chi_{\nu''}(T)\eea
with a sum over the h.w. $\nu'$, resp. $\nu''$, appearing in the decomposition of $\nu\otimes (1,0,1)$,
resp. of $(\nu-\rho)\otimes (0,1,0)$. 
Notice that $(1,0,1)$ is the highest weight of the adjoint representation, hence 
one may write $\nu'=\nu+\balpha$ where $\balpha$
runs over the 12 non zero roots $\balpha$ {for $\nu$ ``deep enough" in the Weyl chamber, 
\ie provided all $\nu+\balpha$ are dominant weights},
and over three times the weight $0$ . 
Thus we may write
\bea \label{exact4}
{\CJ_4(\ell(\lambda+\rho),\ell(\mu+\rho); \ell(\nu+\rho))}
&=& \inv{24}( 9 N_{\lambda\, \mu}^{\nu} +\sum_{\nu'} N_{\lambda\, \mu}^{\nu'}) 
\\ \nonumber
\mathrm{and\ for}\ \nu \mathrm{\ deep\ enough\ in\ } C\quad
&=&\oh (N_{\lambda\, \mu}^{\nu}+\inv{12}\sum_{\balpha} N_{\lambda\, \mu}^{\nu+\balpha})
=
 N_{\lambda\, \mu}^{\nu} +\oh \Delta N_{\lambda\, \mu}^{\nu}
\,.\eea
\\
where $\Delta N_{\lambda\, \mu}^{\nu}:= \inv{12}\sum_{\balpha}( N_{\lambda\, \mu}^{\nu+\balpha}- N_{\lambda\, \mu}^{\nu})$ may be regarded as a second derivative term
(a discretized Laplacian), while the ``first derivative" term vanishes because of $\sum \balpha=0$. \\[5pt]
Example: Take $\lambda=(1,2,2)$, $\mu=(2,2,1)$, $\nu=(1,4,1)$,   the $\nu'$ and their multiplicities read
\bea\nonumber (\nu', N_{\nu\, (1,0,1)}^{\nu'})&=&  \{(0, 3, 2), 1), ((0, 4, 0), 1), ((0, 5, 2), 1), ((0, 6, 0), 
  1), ((1, 3, 3), 1),\\ \nonumber && ((2, 2, 2), 1), ((2, 3, 0), 1), ((2, 4, 2), 
  1), ((2, 5, 0), 1), ((3, 3, 1), 1), ((1, 4, 1), 3)\}\,,\eea 
    $\CJ_4(\ell(\lambda+\rho),\ell(\mu+\rho); \ell(\nu+\rho))=97/24$
  while $N_{\lambda\mu}^\nu=5$, $\sum_{\nu'}N_{\nu\, (1,0,1)}^{\nu'} N_{\lambda\, \mu}^{\nu'}=52$, the rhs
of (\ref{exact4})
  equals $97/24$, and matches the lhs.
  Note that in that example, only 10 out of the 12 $\balpha$ contribute.\\

  There is a second relation, which follows from (\ref{CI-LR1ev}) with the above expression of $\widehat{R}_4$
\be\label{exact4ss}{\CJ_4(\ell(\lambda),\ell(\mu); \ell(\nu))}=\inv{6} \sum_{\nu''}  N_{\lambda-\rho\, \mu-\rho}^{\nu''} N_{\nu-\rho\, (0,1,0)}^{\nu''}\,.\ee
  For the previous example $\lambda=(1,2,2)$, $\mu=(2,2,1)$, $\nu=(1,4,1)$, 
  three weights $\nu''$ contribute $N_{\nu-\rho\, (0,1,0)}^{\nu''}=1$, namely $(0, 2, 0),\  (1, 2, 1),\  (0, 4, 0)$, but 
  only the first two give $N_{\lambda-\rho\, \mu-\rho}^{\nu''} =1$, the third has $N_{\lambda-\rho\, \mu-\rho}^{\nu''} =0$,
  and the rhs equals $\inv{3}$, which is the value of $\CJ_4(\ell(\lambda),\ell(\mu); \ell(\nu))$.

\subsubsection{Case $n=5$}
\label{R5}
For $n=5$, likewise 
\bea \nonumber R_5(T)
&=&\frac{1}{180} \Big[45 +12 \big(\cos(x_1-x_2) +\mathrm{perm.:\ 10\ terms\ in\ total}\big)\\ 
&&\nonumber \qquad
+\big(\cos(x_1+x_2-x_3-x_4) +\mathrm{perm.:\ 15\ terms\ in\ total}\big) \Big]
\\ \nonumber 
&=& \frac{7}{72}+\inv{40} \tr T \tr {T^\star} +\inv{1440}[ (\tr T)^2-\tr T^2][c.c.] \\
\nonumber &=& \frac{7}{72}+\inv{40} \chi_{(1,0,0,0)}(T)  \chi_{(1,0,0,0)}(T^\star) 
+\inv{360}\chi_{(0,1,0,0)}(T) \chi_{(0,1,0,0)}(T^\star)\\ \nonumber 
&=&
\inv{360} \(45 + 10 \chi_{(1,0,0,1)}(T)+  \chi_{(0,1,1,0)}(T)\)\,.
\eea

{Comment: note that at $T=I$, $45+ 10\times 24 + 75=360$, $R_5(I)=1$, as it should.}\\
Then denoting the h.w.   appearing in $(1,0,0,1)\otimes \nu$, resp. $(0,1,1,0)\otimes \nu$, by $\nu'$, 
resp. $\nu''$, 
$$R_5(T) \chi_\nu(T)= \inv{360} \Big(45 \chi_\nu(T) + 10  \sum_{\nu'} \chi_{\nu'}(T)+  \sum_{\nu''} \chi_{\nu''}(T)\Big)$$
and 
\be \label{exact5}
360 \CJ_5(\ell(\lambda+\rho),\ell(\mu+\rho); \ell(\nu+\rho)) =
45 N_{\lambda\, \mu}^{\nu}  +10 \sum_{\nu'}N_{\lambda\, \mu}^{\nu'}+ \sum_{\nu''}N_{\lambda\, \mu}^{\nu''}
\,.\ee
Here again, for $\nu$ ``deep enough" in $C$, we can make the formula more precise:  $\nu'-\nu$ runs over 
the 24 weights (=roots) of the adjoint representation $(1,0,0,1)$, including
 4 copies of $0$ and 20 non zero roots $\balpha$; likewise $\nu''-\nu$ runs over the 75 weights of the 
 $(0,1,1,0)$ representation, including 5 copies of $0$, twice the 20 $\balpha$ 
and  the 30 weights $\hat \beta$ of the form  $\pm(\hat\alpha_{ij}\pm \hat\alpha_{kl})$ with 
$1\le i<j<k<l   \le 5$
or  $\pm(\hat\alpha_{ij} + \hat\alpha_{kl})$ with $1\le i<k<j<l\le 5$. 
Here we are making use of the 
notations $\hat \alpha_i$, $1\le i\le 4$ for  the simple roots, and
$\hat\alpha_{i j}= \hat\alpha_i+\cdots +\hat\alpha_{j-1} $ with   $1\le i<j\le 5$
 for the positive roots .
 Thus ``deep enough"
actually  means: all $\nu+\balpha$ and $\nu+\bbeta \in C$. Then (\ref{exact5}) reads 
\be  \label{exact5b}
{\CJ_5(\ell(\lambda+\rho),\ell(\mu+\rho); \ell(\nu+\rho))}
=N_{\lambda\, \mu}^{\nu} + \inv{30}\sum_{\balpha} (N_{\lambda\, \mu}^{\nu+\balpha}-N_{\lambda\, \mu}^{\nu} )
+\inv{360} \sum_{\bbeta}(N_{\lambda\, \mu}^{\nu+\bbeta} -N_{\lambda\, \mu}^{\nu} )
 \ee 
(with $20/30 + 30/360=3/4$).\\[4pt]
Example. $\lambda=(2, 3, 3, 2)$, $\mu=(3, 2, 3, 2)$, $\nu=(5, 3, 2, 3)$, $N_{\lambda\mu}^\nu=211$.
We find in the lhs of (\ref{exact5}) $360 \CJ_5(\ell(\lambda+\rho),\ell(\mu+\rho); \ell(\nu+\rho)) =63213$
while the three terms in the rhs equal respectively ${9495, 42010, 11708}$ with a sum of 63213, qed.

\subsubsection{Case $n=6$}
\label{R6}
We have found, after long and tedious calculations
\bea\nonumber
2^7 9!\, R_6 &=& 
31356\big( \cos(x_1+x_2+x_3 -x_4-x_5-x_6) +\mathrm{perm.:\ 10\ terms\ in\ total}\big)\\ \nonumber
&&\ + \big( \cos(x_1+x_2+2x_3 -x_4-x_5-2x_6) +\mathrm{perm.:\ 90\ terms\ in\ total}\big)\\ \nonumber
&&\ + 1923\big( \cos(x_1+x_2+x_3 -x_4-2x_5) +\mathrm{perm.:\ 120\ terms\ in\ total}\big)\\ \nonumber
&&\ + 284238 \big( \cos(x_1+x_2-x_3 -x_4) +\mathrm{perm.:\ 45\ terms\ in\ total}\big)\\ \nonumber
&&\ +126 \big( \cos(2x_1+x_2-2x_3 -x_4) +\mathrm{perm.:\ 180\ terms\ in\ total}\big)\\ \nonumber
&&\ + 18906\big( \cos(x_1+x_2-2x_3 ) +\mathrm{perm.:\ 60\ terms\ in\ total}\big)\\ \nonumber
&&\ + 1362\big( \cos(2x_1-2x_2) +\mathrm{perm.:\ 15\ terms\ in\ total}\big)\\ \nonumber
&&\ + 1801128\big( \cos(x_1-x_2) +\mathrm{perm.:\ 15\ terms\ in\ total}\big)\\ \nonumber
&&\ + 4919130\,.\\ \nonumber
\eea

 Alternatively

 \bea\nonumber
 2^8\, 9!\, R_6(T) &=& 
 1699488 + 715852 \chi_{(1,0,0,0,0)}(T) (c.c.)+ 860 
 \chi_{(2,0,0,0,0)}(T) (c.c.)\\ \nonumber &&\ + 
 12032 \Big(\chi_{(0,1,0,0,0)}(T)  \chi^*_{(2,0,0,0,0)}(T)+c.c.\Big)
 + 
 202 683 \chi_{(0,1,0,0,0)}(T) (c.c.) \\ \nonumber &&\ + 
  124 \chi_{(1, 1,0,0,0)}(c.c.)  - 5207 \Big(\chi_{(0,0,1,0,0)}(T)  \chi^*_{(1,1,0,0,0)}(T)+c.c.\Big)  
 + 10414 \chi_{(0,0,1,0,0)}(T)  (c.c.)\\ \nonumber
 &&\ + \chi_{(1,0,1,0,0)}(T)  (c.c.) + 6876 \Big( \chi_{(1,0,1,0,0)}(T) \chi_{(0,1,0,0,0)}(T) +c.c.\Big)
 \\ \nonumber &=& 2629422 \chi_{(0,0,0,0,0)}(T) +
 1670\Big(\chi_{(0,0,1,1,1)}(T) +c.c.\Big)+ 24167 \chi_{(0, 0, 2, 0, 0)}(T)\\ \nonumber &&\ + 
 13826\Big( \chi_{(0, 1, 0, 0, 2)}(T)+c.c.\Big) +216561 \chi_{(0, 1, 0, 1, 0)}(T)+ 957461 \chi_{(1, 0, 0, 0, 1)}(T)
 \\ \nonumber &&\ + 
 \chi_{(1, 0, 2, 0, 1)}(T)+ 125\chi_{(1, 1, 0, 1, 1)}(T)+985 \chi_{(2, 0, 0, 0, 2)}(T)\,.
 \eea
 where the last expression is a decomposition as a sum over real representations, with a total dimension $2^8 9!$, as it should.\\
We also found :
 
  \bea\nonumber
9! \; \widehat{R}_6(T) &=& 
5422 \chi_{(0, 0, 1, 0, 0)}(T) + 
         \chi_{(0, 1, 1, 1, 0)}(T) + 
   13 (\chi_{(0, 2, 0, 0, 1)}(T) + \chi_{(1, 0, 0, 2, 0)}(T)) + \\
   {} & {} &  \label{Rhat6}
 186 \chi_{(1, 0, 1, 0, 1)}(T) + 
 982 (\chi_{(0, 0, 0, 1, 1)}(T) + \chi_{(1, 1, 0, 0, 0)}(T))\,.
 \eea
When evaluated at $T=1$ we check that the dimension count is correct: \\
 $(5422, 1, 13, 186, 982).(20, 1960,  560\times 2, 540, 70\times 2) = 9!$. 
  
We leave it to the reader to write the relations involving $N_{\lambda \mu}^{\nu}$ that follow from (\ref{CI-LR1}) and
 (\ref{CI-LR1ev}), see an example below in sec. \ref{exampleSU6}.
  
 \subsection{Stretching polynomials}
\label{LRpoly}
 
\subsubsection{The case $n=2$}
This is a trivial case. Since for any admissible triple, $N_{\lambda\mu}^\nu=1$, we have, 
 according to a general result \cite{KTT04},  $P_{\lambda\mu}^\nu(s)=1$.
\subsubsection{The case $n=3$}
For $n=3$, we have, from point 2. in sec.\,\ref{caseSU3}
$$ N_{\lambda\mu}^\nu-1=
\CJ_3(\ell(\lambda),\ell(\mu); \ell(\nu))$$
and the latter is an homogeneous linear function of $s$, hence
\be P_{\lambda\mu}^\nu(s)=N_{s\lambda\, s\mu}^{s\nu}=\CJ_3(\ell(\lambda),\ell(\mu); \ell(\nu))+1
= (N_{\lambda\mu}^\nu-1)s +1\,.\ee
This expression is also valid for weights $\lambda$ and/or $\mu$ on the boundary of the Weyl chamber $C$,
in which case, as is well known (``Pieri's rule"), all LR multiplicities equal 1, and then 
again by the same general result \cite{KTT04}, 
$P_{\lambda\mu}^\nu(s)=1$, while as noticed above, $\CJ_3=0$. Likewise as noticed  in sec.\,2.2.2,
if $\nu$ lies on the boundary of tensor polytope, (the outer matriochka), $N_{\lambda\, \mu}^{\nu}=1$  
 and thus again, $P_{\lambda\mu}^\nu(s)=1$.

 Remark. The property that $P_{\lambda \mu}^{\nu} (s) = 1 + s (N_{\lambda \mu}^{\nu} -1)$ had been
 proved in \cite{KTT04}, then recovered in \cite{Rass} using vector partition functions.

\subsubsection{The case $n=4$}
\label{stretch2}

For $n=4$,  given weights $\lambda,\mu \notin \partial C$, 
and  weights $\nu$ {\it interior} to the polytope, 
$\CJ_4(\ell(\lambda)\ell(\mu); \ell(\nu)))\ne 0$  (assuming that Lemma \ref{conj1} holds true) and
the stretching polynomial $P_{\lambda\mu}^\nu(s)$ is of degree exactly 3. 
 Now let us Taylor expand
$$ \CJ_4( \ell(s\lambda+\rho),\ell(s\mu+\rho); \ell(s\nu+\rho))=
s^3\CJ_4(\ell(\lambda),\ell(\mu); \ell(\nu)) + \oh s^2 {a} +O(s)\,,$$
where the coefficient
 $a$, stemming here from the first order derivatives of $\CJ_4$, will receive shortly a geometric interpretation.

The stretching polynomial $P_{\lambda\mu}^\nu(s) $ must satisfy the three conditions
\begin{enumerate}
\item $P_{\lambda\mu}^\nu(1)=N_{\lambda\mu}^\nu$, by definition; 
\item $P_{\lambda\mu}^\nu(0)=1$;  
\item $P_{\lambda\mu}^\nu(s)= \CJ_4(\ell(\lambda),\ell(\mu); \ell(\nu)) s^3+   
 \oh s^2 { a} +O(s)$, as discussed in (\ref{asymp2}).
\end{enumerate}
Recall now the discussion of sec.\,\ref{Epols} and \ref{LRpols} : $\CJ_4(\ell(\lambda),\ell(\mu); \ell(\nu))$ is $\inv{6}$ times
the normalized volume $\CV$  of the hive polytope, and $a$  is  
half the total normalized area $\CA$. 
There is a unique polynomial satisfying these conditions, namely 
\bea \nonumber
P_{\lambda\mu}^\nu(s) &=&\CJ_4(\ell(\lambda),\ell(\mu); \ell(\nu))   s^3+  
\inv{4} \CA s^2+ 
\Big( N_{\lambda\mu}^\nu - \CJ_4(\ell(\lambda),\ell(\mu); \ell(\nu)) - 
\inv{4} \CA  
-1\Big)  s+1\\
\label{polstr4}  &=&  \inv{6} \CV \,s^3 +\inv{4} \CA s^2 + (N_{\lambda\mu}^\nu  -  \inv{6} \CV - \inv{4} \CA  -1)s +1\,.\eea

 Then the alleged non-negativity of the $s$ coefficient \cite{KTT04}  
  amounts to 
\be\label{ineq1} N_{\lambda\mu}^\nu  \,{ \buildrel  ? \over{\ge } }\,\inv{6} \CV  +\inv{4} \CA +1\,,\ee
while the counting of interior points, through Ehrhart--Macdonald reciprocity theorem, gives
us another lower bound on $ N_{\lambda\mu}^\nu$
$$ \#\mathrm{(interior\ points)} =- P_{\lambda\mu}^\nu(-1)= N_{\lambda\mu}^\nu-(\oh \CA +2)\ge 0\,. $$

In \cite{BMM,BDLDPS}
  inequalities were obtained  between coefficients of the Ehrhart polynomial of an
 integral polytope.  Recall that for $n=4$,
 all hive polytopes are integral \cite{Buch}, and we may apply on (\ref{polstr4}) these inequalities  which read 
 \bea \nonumber  \frac{\CA}{4} &\le& \frac{\CV}{2}+\oh\\  \nonumber
N_{\lambda\mu}^\nu  -  \inv{6} \CV - \inv{4} \CA  -1 &\le& \frac{\CV}{3} +\frac{3}{2}
 \eea
  hence 
 \be N_{\lambda\mu}^\nu  \le \frac{\CV}{2} + \frac{\CA}{4} +\frac{5}{2} \le \CV +3\ee
 which is precisely the  Blichfeldt inequality mentioned above at point 4 of sec. \ref{case4}.
 \\[5pt]

 In contrast, for non generic triples $(\lambda,\mu;\nu)$,
 $\CJ_4(\ell(\lambda),\ell(\mu); \ell(\nu))  =0$, the stretching 
polynomial is of degree strictly less than 3, and reads in general
 \be \label{degcase} P_{\lambda\mu}^\nu(s) =\frac{a}{2} s^2 +( N_{\lambda\mu}^\nu -\frac{a}{2}  -1) s +1\,.\ee
 If the coefficient $a$  is non vanishing, it has
now to be interpreted as the normalized area of the 2-dimensional hive polytope (a 
polygon). If $a=0$, either $N_{\lambda\mu}^\nu\ge 2$  and $P_{\lambda\mu}^\nu(s) = (N_{\lambda\mu}^\nu-1)s+1$,
or $N_{\lambda\mu}^\nu=1$ and $P_{\lambda\mu}^\nu(s) =1$, 
consistent with the result of sec.\,\ref{subleadingterms} and the two general results $P=1$ if $N_{\lambda\mu}^\nu=1$ and $P=s+1$ if $N_{\lambda\mu}^\nu=2$.  \\
In the former case (dimension 2 polytope, degree 2 Ehrhart polynomial), 
Erhrart--Macdonald reciprocity theorem gives us 
an {\it upper} bound on $N_{\lambda\mu}^\nu \le a +2$,  while the alleged  non-negativity of the $s$-coefficient gives a 
lower bound, $N_{\lambda\mu}^\nu \ge\oh( a +2)$. Thus one should have
\be\label{ineq2} \oh( a +2)\,  \buildrel  ? \over{\le } \, N_{\lambda\mu}^\nu \le a +2\,.\ee
Also denoting $ c:=\# {\rm internal\ points\ }= P(-1)= a -N_{\lambda\mu}^\nu +2$,
$b=\# $ boundary points, $b+c:=\# {\rm total \ of \ points\ }= N_{\lambda\mu}^\nu$, hence $a+2= b+2c$
which is Pick's formula for the Euclidean area  $a/2= b/2 + c-1$.
 \\[5pt]

\def\CJ{{\mathcal J}}
\noindent {\bf Examples}:  Here
we denote for short  $\CJ'_4=\CJ_4( \ell(s\lambda+\rho),\ell(s\mu+\rho); \ell(s\nu+\rho))$. \\[4pt]
Take $\lambda=(2,2,1)$, $\mu=(2,1,3)$, \\
for $\nu=(0,1,4)$, $N_{\lambda\mu}^\nu=3$, $P_{\lambda\mu}^\nu(s) =\oh(s+1)(s+2)$, 
$\CJ'_4=\inv{12}(6s^2+15 s +7)$\\
 while
for $\nu=(2,4,0)$, $N_{\lambda\mu}^\nu=3$, $P_{\lambda\mu}^\nu(s) =2s+1$,  $ \CJ'_4=
\oh(1+4s)$\\
  and for $\nu=(2,0,4)$, $N_{\lambda\mu}^\nu=4$, $P_{\lambda\mu}^\nu(s) = (s+1)^2$, 
  $ \CJ'_4=\inv{4}(4s^2+7s+2)$. \\[5pt]
    Take $\lambda=(3,0,3)$, $\mu=(2,3,1)$, \\ for $\nu=(3,4,0)$,  $N_{\lambda\mu}^\nu=3$, $P_{\lambda\mu}^\nu(s) =2s+1$,
  $ \CJ'_4=\inv{8} (14s+5)$, \\
  while for $\nu=(2,3,1)$, $N_{\lambda\mu}^\nu=6$, $P_{\lambda\mu}^\nu(s) =(s+1)(2s+1)$, $\CJ'_4=\inv{8}(16 s^2+18s+3)$.

\subsection{The hive polytope: three examples}
\subsubsection{An example in SU(4)}
\label{exampleSU4}

\begin{figure}[bthp]
\centering{
\includegraphics[width=15pc]{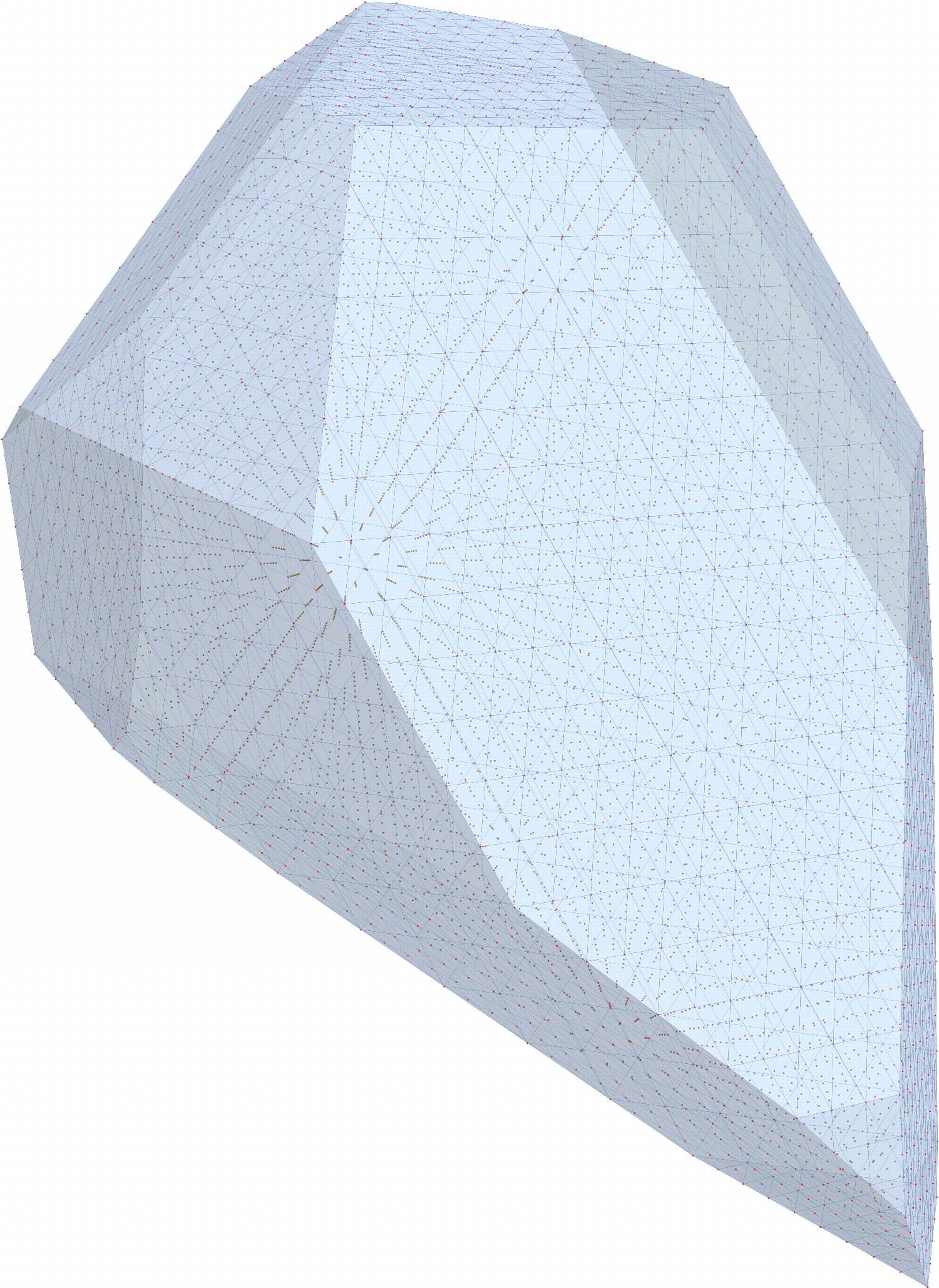}
\qquad
\includegraphics[width=20pc]{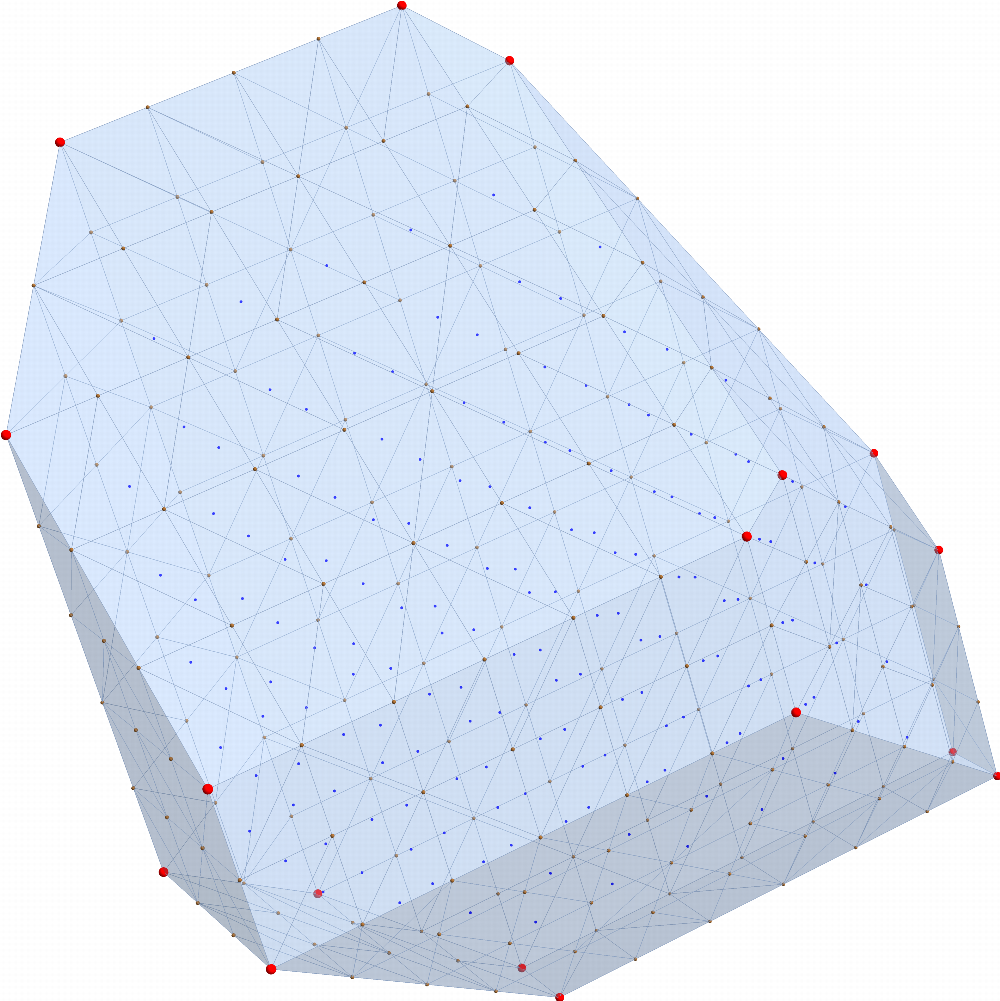}
\caption{\label{SU4Example} 
Left: The  $\SU(4)$ tensor polytope
{${\mathbf H}_{\lambda\mu}$} for $\lambda= (21,13, 5)$,  $\mu= (7,10, 12)$ ,  and its $7092$ integral points (distinct irreps).
   Each such point can itself be thought as a hive polytope, for example the one given on the right.\\
Right: The $\SU(4)$ hive polytope {${\mathcal H}_{\lambda\mu}^\nu$} 
associated with the
branching rule:  $((21,13, 5), (7,10, 12); (20,11,9))$.
Each integral point (367 of them)  stands for a pictograph describing an allowed coupling of this triple, for example the one given in fig.~\ref{apictograph}.}
}
\end{figure}

{
Consider the irreps of highest weight $\lambda= (21,13, 5)$ and $\mu= (7,10, 12)$.
Their tensor product contains $7092$ distinct irreps $\nu$ with  multiplicities ranging from $1$ to $377$. 
The tensor polytope ${\mathbf H}_{\lambda\mu}$ is displayed in fig.~\ref{SU4Example}, left.
The total multiplicity (sum of multiplicities for the various $\nu$'s) is $537186$.

Let us now consider
a particular term in  the decomposition of the tensor product into irreps: the admissible triple $(\lambda,\mu; \nu)$, with $\nu=(20,11,9)$, whose multiplicity is equal to $367$.
This term can be thought  of as a particular point of the 
tensor polytope
and stands itself for a hive polytope of dimension $3$ ($d=(n-1)(n-2)/2=3$ for $\SU(4)$).
It is displayed in fig.~\ref{SU4Example}, right.
It has $367$ integral points: $160$ are interior points, in blue in the figure, and $207$ are boundary points. Among the latter, $17$ are vertices, in red in the figure, the other boundary points are in brown.
The polytope is integral since its vertices are integral -- it is always so for SU(4) (see \cite{Buch}, example 2). 
Every single one of the $367$ points of the polytope displayed in fig.~\ref{SU4Example}, right, 
stands for a pictograph contributing by $1$ to the multiplicity of the chosen tensor product  branching rule.
For illustration, we display one of them on fig.~\ref{apictograph}; actually we give  several versions of this pictograph: 
first, the isometric honeycomb version and its dual, the O-blade version,  and then, the KT-honeycomb version and its corresponding hive.
Notice that for the first two kinds of pictographs the external vertices are labelled by Dynkin components of the highest weights, whereas for the last two, they are labelled by  Young partitions.

The hive polytope has $12$ facets (eight quadrilaterals, three pentagons and one heptagon), $27$ edges, and $17$ vertices (and Euler's identity is satisfied: $12-27+17=2$). \\
 Its normalized volume and
area are ${\mathcal V} = 1484$ and  ${\mathcal A} = 410$.}\\
The number of pictographs with prescribed edges gives the following sequence of multiplicities\\
$N_{s \lambda\, s\mu}^{s \nu}=  \{367, 2422, 7650, 17535, 33561, 57212, 89972, 133325, 188755, 257746, \ldots\}$,  for $s=1,2\ldots$
Only the first three terms of this sequence are used to determine the LR polynomial if we impose that its constant term be equal to 1:
$P_{\lambda, \mu}^{\nu}(s)=  (5936 s^3 + 2460 s^2 + 388 s + 24)/4!$\; 
From our discussion in sec.~\ref{Epols}, $P_{\lambda, \mu}^{\nu}(s)$ should be equal to the Ehrhart polynomial $E(s)$ of the hive polytope; using the computer algebra package Magma \cite{Magma} we checked that it is indeed so.
\vskip 0.1cm
{ The direct calculation of $\CJ_4$ using  (\ref{CI4}) gives
$\CJ_4(\ell(\lambda),\ell(\mu); \ell(\nu))=742/3$, 
and more generally $\CJ_4( \ell(s \lambda),\ell(s \mu); \ell(s\nu))=742 \, s^3/3$.
Using the same eq. (\ref{CI4}), we can also calculate $\CJ_4$ for $\rho$-shifted arguments: $\CJ_4( \ell(s \lambda+\rho),\ell(s \mu+\rho); \ell(s\nu+\rho)) =   \frac{742 }{3} s^3 + \frac{205}{2}s^2  +12 s+ \frac{5}{12}$.
In agreement with our general discussion {of sec.~\ref{subleadingterms}}, the  first two terms
of $ P_{\lambda \mu}^{\nu}(s)$ and of 
$ \CJ_4( \ell(s \lambda+\rho),\ell(s \mu+\rho); \ell(s\nu+\rho))$ are identical, the leading term being  also equal to $\CJ_4( \ell(s\lambda),\ell(s\mu); \ell(s\nu))$.
One checks that the leading coefficient of $E(s)$, hence of $P_{\lambda, \mu}^{\nu}(s)$, is equal to $\frac{1}{3!}$ of
the normalized
volume of the polytope and that
the second coefficient is equal to  $\inv{2}\inv{2!}$ of the normalized
2-volume of its boundary.
In accordance with  Ehrhart--Macdonald reciprocity theorem, one also checks  that $-P_{\lambda \mu}^{\nu}(-1)=160$, the number of interior points in the polytope.
Finally, on this example, one can test eq (\ref{exact4}) which relates 
$\CJ_4( \ell(\lambda+\rho),\ell(\mu+\rho); \ell(\nu+\rho))=1449/4$ 
to a sum of the Littlewood-Richardson coefficient $N_{\lambda \mu}^\nu$ and  its twelve ``neighbors" $\nu+\hat\alpha$ appearing in the tensor product $\nu\otimes (1,0,1)$. 
Likewise eq (\ref{exact4ss})  relates $6 \CJ_4(\ell(\lambda),\ell(\mu); \ell(\nu))=1484$ 
to a sum over six weights $\nu''=(18, 10, 9),(18, 11, 7),(19, 9, 8),(19, 11, 8),(20, 9, 9),(20, 10, 7)$ of
 the product $N_{\lambda-\rho\,\mu-\rho}^{\nu''} N_{\nu-\rho\, (0,1,0)}^{\nu''}$   
 which takes the respective values ${254, 235, 254, 243, 259, 239}$,  the sum being indeed $1484$.

\begin{figure}[!tbp]
 \centering
 \begin{minipage}[b]{0.4\textwidth}
  \includegraphics[width=16pc]{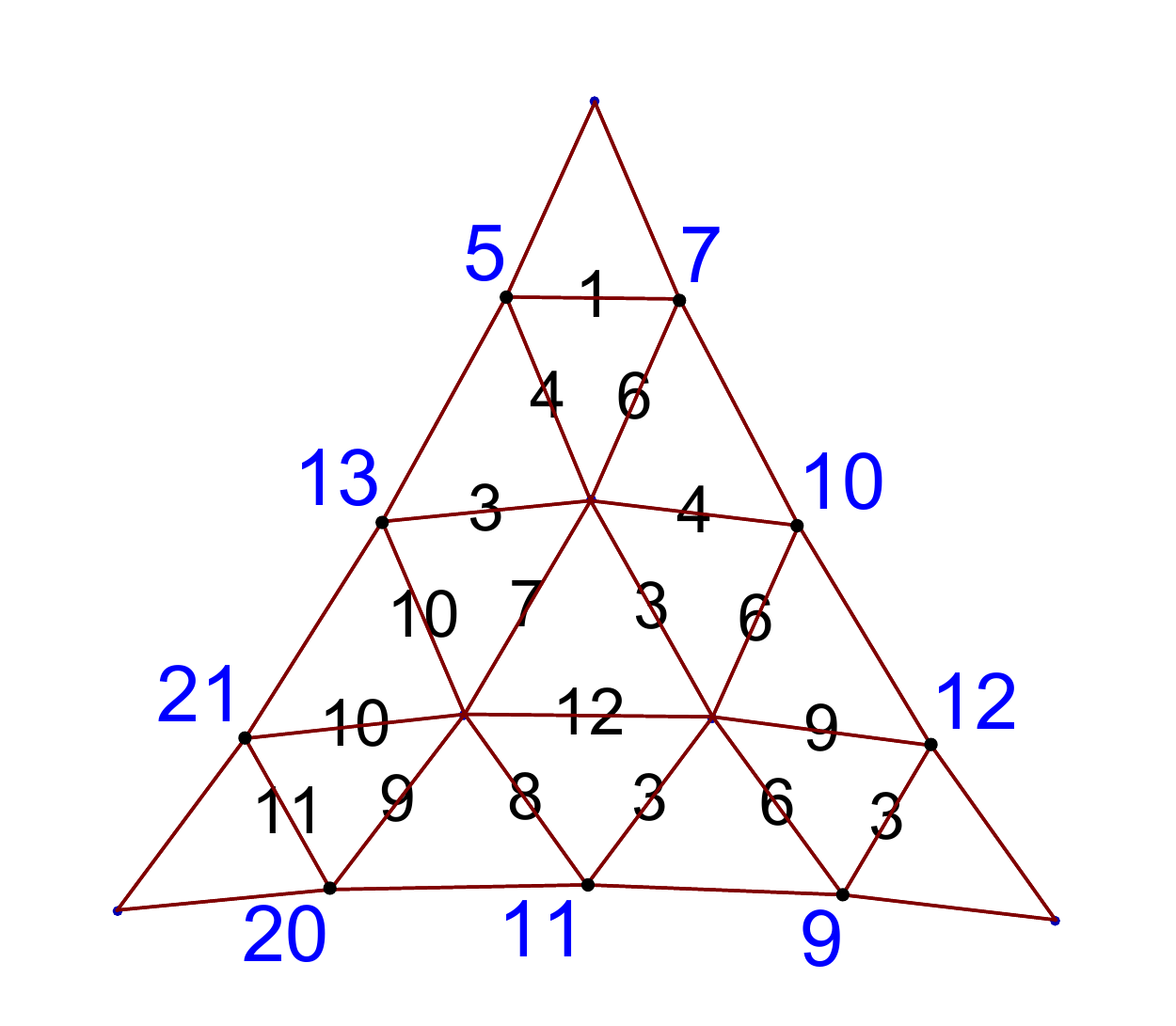} 
   \caption*{O-blade version: edges are non-negative integers, opposite angles (sum of adjacent edges) around the inner points are equal.}
 \end{minipage}
 \hfill
 \begin{minipage}[b]{0.4\textwidth}
   \includegraphics[width=16pc]{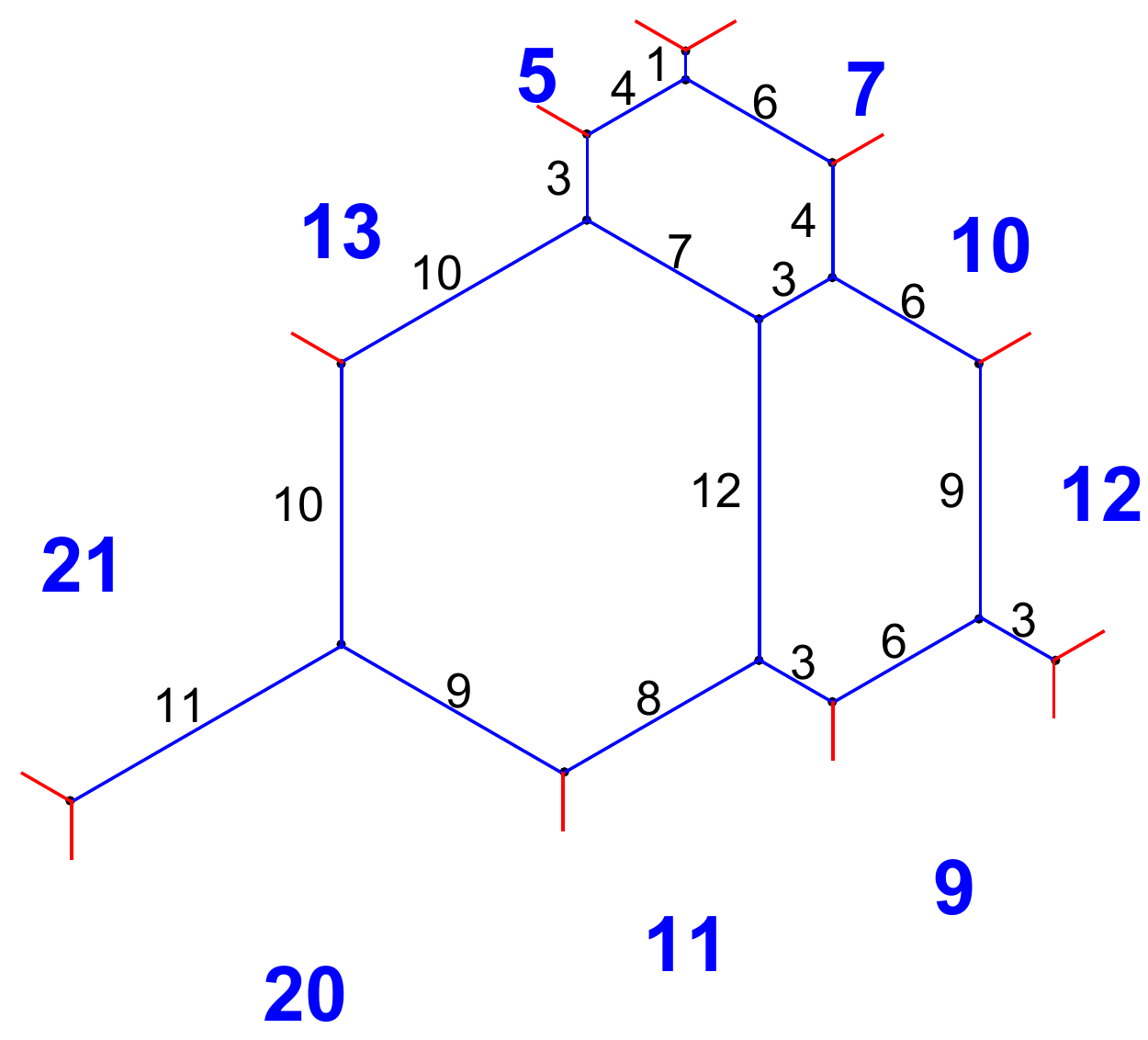}
   \caption*{Isometric honeycomb version: opposite angles (sum of adjacent edges) of hexagons are equal.}
 \end{minipage}
 \includegraphics[width=15pc]{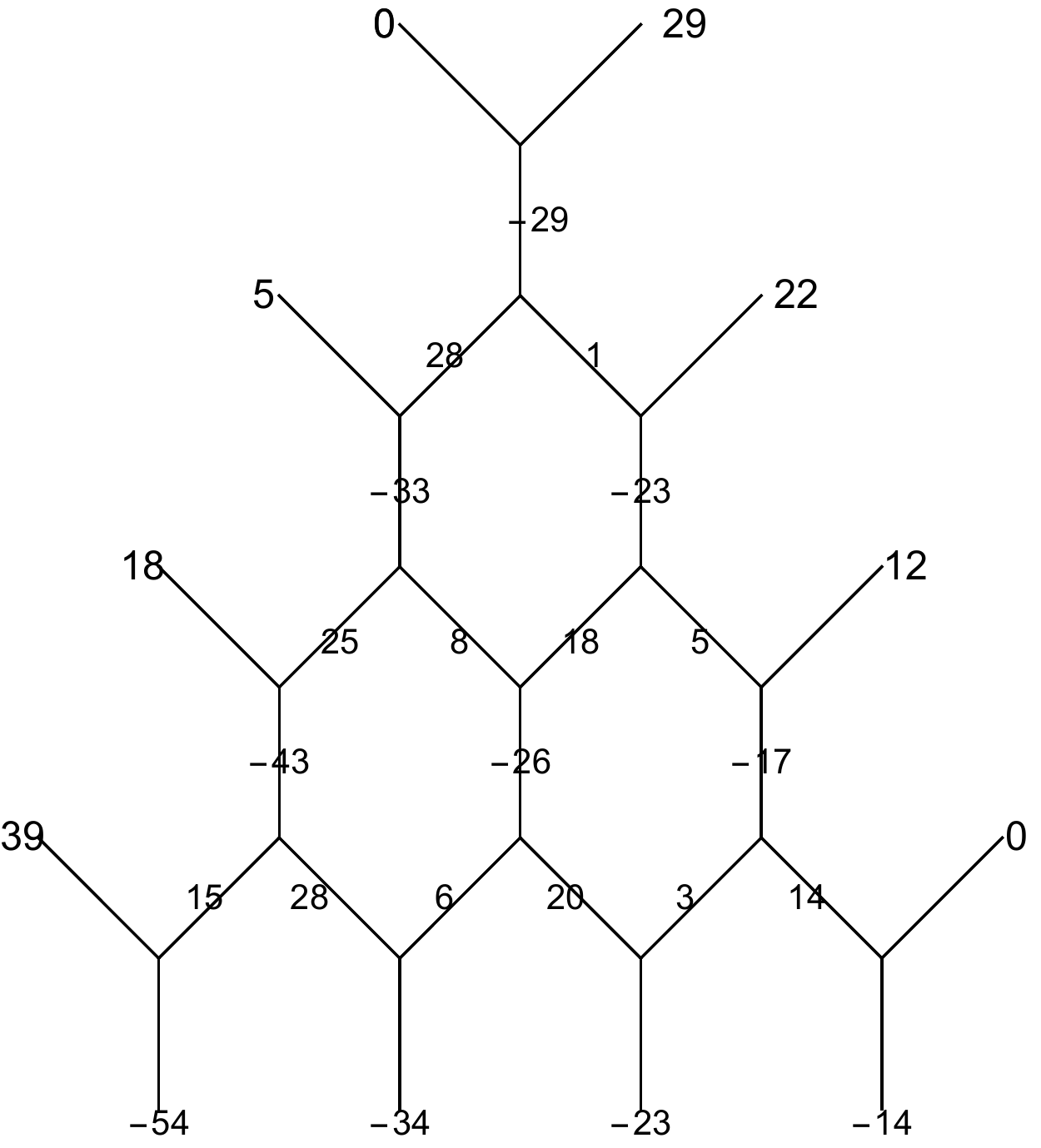}\hskip35mm
 \raisebox{7ex}{\includegraphics[width=13pc]{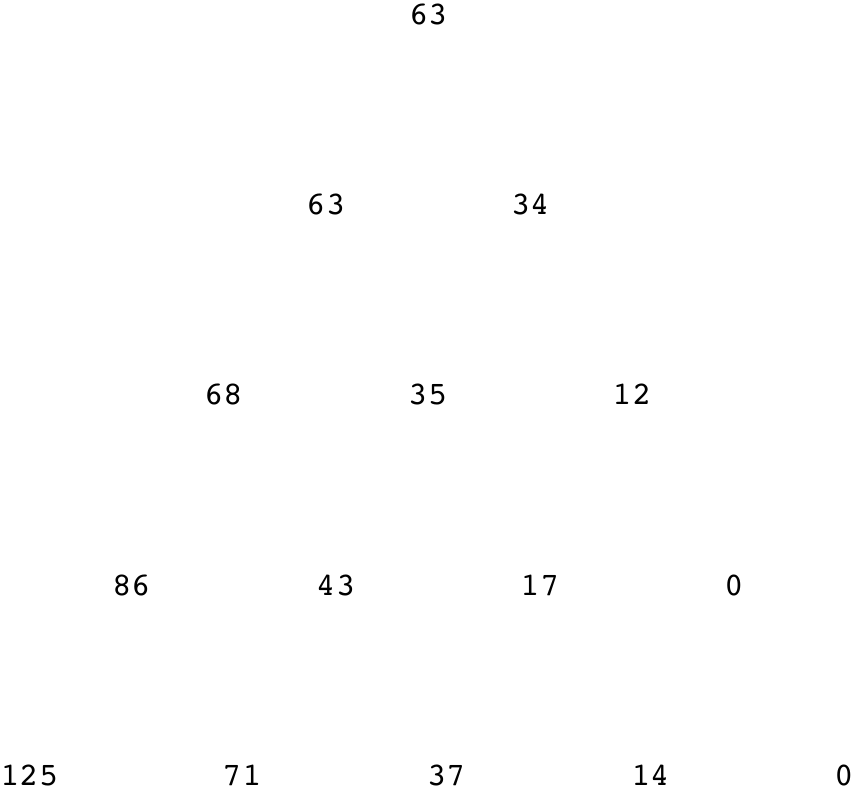}}
 \label{isometriX} 
 \caption{\label{apictograph}One of the 367 pictographs associated with $((21,13, 5), (7,10, 12); (20,11,9))$.
  For completeness we also give below the corresponding KT-honeycomb and its dual  hive.
}
\end{figure}

\subsubsection{An example in SU(5)}
\label{exampleSU5}
Consider the following tensor branching rule of $\SU(5)$: $(\lambda,\mu; \nu)$ with $\lambda=(1,3,2,3)$,  $\mu=(2,1,4,2)$, $\nu=(3,1,4,3)$.
The hive polytope $\CH_{\lambda\mu}^\nu$  has dimension $d=6$. We shall see that it is {\sl not \/} an integral polytope.
We denote ${\mathcal Q}$   the convex hull of its integral points.}
$\CH_{\lambda\mu}^\nu$  has $66$ vertices and $99$ points, all of them being boundary points.
${\mathcal Q}$ has $64$ vertices and $99$ points (the latter being the same as for $\CH_{\lambda\mu}^\nu$, by definition).
Therefore we see that $2$ vertices of $\CH_{\lambda\mu}^\nu$ are not (integral) points of $\CH_{\lambda\mu}^\nu$.
The normalized volume of $\CH_{\lambda\mu}^\nu$ is $2544$ (it is 2538 for ${\mathcal Q}$). 
The normalized volume of the boundary of $\CH_{\lambda\mu}^\nu$  is $3630$ (it is $3618$ for 
${\mathcal Q}$). 
The LR polynomial  $P_{\lambda\mu}^\nu(s)$, \ie the Ehrhart 
polynomial of $\CH_{\lambda\mu}^\nu$, is
$53 s^6/15 + 121 s^5/8  + 667 s^4/24 + 679 s^3/24  + 687 s^2/40 + 73 s/12 +1  $.
In the case of $\CH_{\lambda\mu}^\nu$ , we check the first two coefficients related to the $6$-volume of the polytope and to the $5$-volume of the facets: $2544/6! = 53/15$ and $1/2 \times 3630 / 5! = 121/8$.
The Ehrhart polynomial of
${\mathcal Q}$ is $141 s^6/40 + 603 s^5/40 + 665 s^4/24 + 679 s^3/24 + 259 s^2/15+ 92 s/15  +1$.
In the case of ${\mathcal Q}$,
the same volume checks read: $2538/6! = 141/40$ and $1/2 \times  3618 / 5! = 603/40$.

An independent calculation using the function $\CJ_5$
gives
$\CJ_5(\ell(\lambda),\ell(\mu); \ell(\nu))= 53/15$, 
the leading coefficient of the stretching polynomial.

In the present example, where $\CH_{\lambda\mu}^\nu$  and  
${\mathcal Q}$ differ, it is instructive to consider what happens under scaling.
The two vertices of $\CH_{\lambda\mu}^\nu$  that are not integral points are actually half-integral points, so that they become integral by doubling.
The polytope $2 \CH_{\lambda\mu}^\nu$ has again $66$ vertices (by construction), it is integral, it has 1463 points, 18 being interior points and 1445 being boundary points.
It could also be constructed as the hive polytope  associated with the doubled branching rule $(2\lambda, 2\mu; 2\nu)$, and its own Littlewood-Richardson (LR) polynomial, equal to its Ehrhart polynomial, can be obtained from the LR polynomial of $ \CH_{\lambda\mu}^\nu$ by substituting $s$ to $2s$.\\
The polytope $2{\mathcal Q}$ 
has again $64$ vertices (of course), it is integral, it has 1460 points, 18 being interior points ans 1442 being boundary points.
Since ${\mathcal Q} 
\subset  \CH_{\lambda\mu}^\nu$ we have 
$2{\mathcal Q} \subset  2\CH_{\lambda\mu}^\nu$, but now both polytopes are integral (and they are different).\\
${\mathcal Q}$ and $\CH_{\lambda\mu}^\nu$  have the same integral points, so, in a sense, they describe the same multiplicity for the chosen triple $(\lambda,\mu; \nu)$, 
however, under stretching (here doubling) of the branching rule, we have to consider $2\CH_{\lambda\mu}^\nu$, not 
$2{\mathcal Q}$, otherwise we would miss three honeycombs ($= 1463 - 1460$) and find an erroneous multiplicity.
These three honeycombs correspond to the two (integral) vertices of $2\CH_{\lambda\mu}^\nu$ coming from the two (non integral) vertices of $\CH_{\lambda\mu}^\nu$  that became integral under doubling, plus one extra (integral) point, which is a convex combination of vertices. For illustration purposes we give below the three pictographs (in the O-blade version) that correspond to these three points.

\begin{figure}[bthp]
\centering{
\includegraphics[width=40pc]{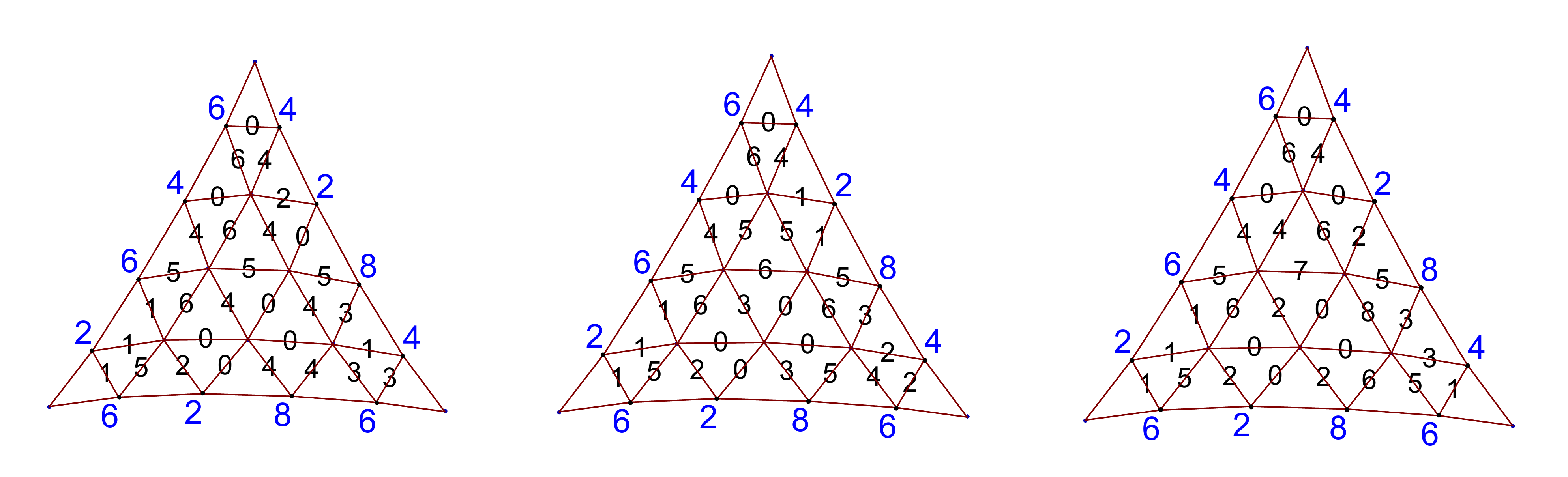} 
\caption{\label{obladeSU5troisgraces} 
The three $\SU(5)$ pictographs (O-blade version) associated with $(2 \lambda, 2 \mu; 2 \nu)$, with $\lambda= (1,3, 2, 3) $, $\mu= (2, 1, 4, 2)$, $\nu= (3, 1, 4, 3)$
that belong to the hive polytope  of this doubled branching rule but that do not belong to the double of the integral part of the hive polytope  of $(\lambda,\mu;  \nu)$.}
}
\end{figure}

\subsubsection{An example in SU(6)}
\label{exampleSU6}
We consider the following tensor branching rule of $\SU(6)$: $(\lambda,\mu; \nu)$ with $\lambda=(1,3,1,2,1)$,  $\mu=(2,1,3,2,1)$, $\nu=(4,1,6,2,1)$. The multiplicity is $38$.

For SU(6), the number of fundamental pictographs is $5\times 3+2\times 10$ but there are $10$ syzygies (one for each inner hexagon in the honeycomb picture) so that a basis has $25$ elements,  the set of $38$ (integral) honeycombs is then described as a $25 \times 38$ matrix. The convex hull of these $38$ points is then calculated, one finds that it is a $10$ dimensional polytope ${\mathcal Q}$ (in $\R^{25}$). The obtained polytope --which has no interior point and $38$ integral points, $36$ of them being vertices-- happens {\sl not} to coincide with the hive polytope ${\mathcal H}$ (we are in a situation analogous to the one examined in the previous SU(5) example).  A quick study of  ${\mathcal Q}$ reveals that this polytope, and so ${\mathcal H}$ itself, has dimension $10$,  and that the chosen triple is therefore generic.\\
The fact that ${\mathcal H}$ differs from ${\mathcal Q}$ can be seen in (at least) three different ways:  1) The Ehrhart polynomial of ${\mathcal Q}$ fails to recover the multiplicity of $(s\,\lambda,s\,\mu; s\,\nu)$, already for $s=2$ where the multiplicity is $511$.  2) The leading coefficient ($30/9!$) of this polynomial, hence the normalized volume of ${\mathcal Q}$,  differs from $\CJ_6(\ell(\lambda),\ell(\mu); \ell(\nu))=32/9!$ determined {directly or} from Theorem \ref{CI-LR} (part 2), we shall come back to this below. 3) A direct determination of the polytope ${\mathcal H}$ obtained as an intersection of $45$ half-spaces --interpreted for instance as the number of (positive) edges in the oblade picture-- will show that ${\mathcal H}$ is not an integral polytope (its vertices, aka corners, are rational but not all integral) and its integral part is indeed  ${\mathcal Q}$. We leave this as an exercise to the reader.
The LR-polynomial associated with the chosen triple, equivalently the Ehrhart polynomial of ${\mathcal H}$, is equal to 
$$
\frac{s^{10}}{11340}+\frac{67 s^9}{24192}+\frac{899 s^8}{24192}+\frac{5639 s^7}{20160}+\frac{11281 s^6}{8640}+\frac{22763 s^5}{5760}+\frac{572777
   s^4}{72576}+\frac{78481 s^3}{7560}+\frac{88351 s^2}{10080}+\frac{3683 s}{840}+1
$$
while the Ehrhart polynomial of ${\mathcal Q}$ is
$$\frac{s^{10}}{12096}+\frac{947 s^9}{362880}+\frac{203 s^8}{5760}+\frac{3235 s^7}{12096}+\frac{227 s^6}{180}+\frac{66767 s^5}{17280}+\frac{946187
   s^4}{120960}+\frac{94585 s^3}{9072}+\frac{1421 s^2}{160}+\frac{11189 s}{2520}+1\,.$$
The coefficient of $s^{10}$, equal to $1/11340=32/9!$ and interpreted as the normalized volume of ${\mathcal H}$, 
can be obtained from a direct evaluation of the expression of $\CJ_6$, but it can also be obtained easily from Theorem \ref{CI-LR} (part 2).  This double sum (\ref{CI-LR1ev0})   involves the seven weights $\kappa$ together with the seven associated coefficients $\widehat{r_\kappa}$ that appear in (\ref{Rhat6}) 
and turns out to involve only the following weights $\nu^\prime$: $(1,2,2,2,0),(1,2,3,0,1),(2,1,2,1,1), (2,1,3,0,0)$.
Most terms are actually zero (because of the vanishing of many Littlewood-Richardson coefficients), and the result is $(1+2+2+1+13+13)/9! = 32/9!$.

 \section*{Acknowledgements} 
We acknowledge stimulating discussions with 
  Olivier Babelon,  Paul Zinn-Justin  and especially Allen Knutson and Mich\`ele Vergne.

\end{document}